\documentclass[a4paper,english]{article}

\usepackage{babel}
\usepackage{lastpage}
\usepackage{amssymb,amsmath}
\usepackage{graphicx}
\usepackage[sf,SF]{subfigure}
\usepackage{enumitem}
\usepackage{caption}
\usepackage{units}
\usepackage{amsthm,euscript}
\usepackage[a4paper, portrait, margin=1in]{geometry}
\usepackage[bookmarks=true,hidelinks]{hyperref}

\newtheorem*{maintheorem}{Main theorem}
\newtheorem{theorem}{Theorem}
\newtheorem{proposition}[theorem]{Proposition}

\newtheorem{definition}[theorem]{Definition}
\newtheorem{lemma}[theorem]{Lemma}

\numberwithin{equation}{section}
\numberwithin{theorem}{section}
\numberwithin{table}{section}

\newcommand{\R}{\mathbb{R}}
\newcommand{\N}{\mathbb{N}}
\newcommand{\Z}{\mathbb{Z}}
\newcommand{\Dx}{{\Delta x}}
\newcommand{\Dt}{{\Delta t}}
\renewcommand{\leq}{\leqslant}
\renewcommand{\geq}{\geqslant}
\renewcommand{\phi}{\varphi}
\newcommand{\Lip}{\mathrm{Lip}}
\newcommand{\DLip}{\mathrm{DLip}}
\newcommand{\define}[1]{\textbf{#1}}
\newcommand{\ind}{\chi}
\newcommand{\hf}{{\unitfrac{1}{2}}}
\renewcommand{\epsilon}{\varepsilon}

\allowdisplaybreaks

\title{Second-order convergence of \\monotone schemes for conservation laws} 
\author{Ulrik S. Fjordholm \and Susanne Solem}
\date{\today}
\begin{document}
\maketitle
\begin{abstract}
We prove that a class of monotone, \emph{$W_1$-contractive} schemes for scalar conservation laws converge at a rate of $\Dx^2$ in the Wasserstein distance ($W_1$-distance), whenever the initial data is decreasing and consists of a finite number of piecewise constants. It is shown that the Lax--Friedrichs, Enquist--Osher and Godunov schemes are $W_1$-contractive. Numerical experiments are presented to illustrate the main result. To the best of our knowledge, this is the first proof of second-order convergence of any numerical method for discontinuous solutions of nonlinear conservation laws.
\end{abstract}

{ \small \textbf{Keywords.}  Hyperbolic conservation laws, numerical methods, convergence rate, Wasserstein metric. }

{ \small \textbf{Mathematics Subject Classification.} 65M08, 65M12 }

\section{Introduction}
Motivated by numerical results, we show that the class of so-called \emph{$W_1$-contractive}, monotone finite volume schemes for the scalar conservation law 
\begin{equation}\label{eq:cons_law}
\begin{split}
u_t +f(u)_x = 0,& \qquad x \in \R, \ t >0, \\
u(x,0) = u_0(x)&
\end{split}
\end{equation}
with \emph{decreasing} and \emph{piecewise constant} initial data $u_0(x)$, will converge to the entropy solution with a second-order convergence rate in the Wasserstein distance $W_1$. 

\begin{maintheorem}
Let $f$ be convex and let $u_0$ be piecewise constant and decreasing. Then any monotone $W_1$-contractive finite volume scheme will converge to the exact solution of \eqref{eq:cons_law} at a rate of $\Dx^2$, as measured in the Wasserstein distance.
\end{maintheorem}

\noindent
The full theorem is stated in Section \ref{sec:mainthm}. As is well known, the entropy solution for this type of initial data will solely consist of shocks moving at constant speeds. Thus, the entropy solutions considered in this work constitutes a simple, but fundamental class of solutions for \eqref{eq:cons_law}.

\subsection{The Wasserstein and $\Lip'$ distances}
The Wasserstein distance $W_1$ (also called the Kantorovich--Rubinstein metric) is a metric on the set of probability measures on $\R$ (see \cite{Vil03} for further details), and can be thought of as measuring the amount of work required to ``move mass'' from one probability measure to another; see Figure \ref{fig:wasserstein}(a). According to the Kantorovich--Rubinstein duality theorem (see Rachev and Shortt \cite[Theorem 2.6]{rachev_90}), the Wasserstein distance between two probability measures $\mu$ and $\nu$ on $\R$ can equivalently be defined as
\begin{equation}\label{eq:w1norm}
W_1(\mu,\nu) := \sup_{\|\phi\|_\Lip\leq 1} \int_{\R} \phi(x)\ d(\mu-\nu)(x).
\end{equation}
Here, the supremum is taken over all functions $\phi:\R\to\R$ with Lipschitz semi-norm $\|\phi\|_\Lip := \sup_{x\neq y}\Bigl|\frac{\phi(y)-\phi(x)}{y-x}\Bigr|$ at most 1. Although normally only defined for \emph{probability} measures, the right-hand side of \eqref{eq:w1norm} is well-defined and finite as long as the difference $\mu-\nu$ has mass $(\mu-\nu)(\R)=0$, and decays sufficiently fast as $x\to\pm\infty$. Given Borel measurable functions $u,v:\R\to\R$ satisfying the analogous properties
\begin{subequations}\label{eq:wasserdef}
\begin{equation}\label{eq:wassercondition}
\int_\R (u-v)(x)\ dx = 0, \qquad \int_\R |x|\, |u-v|(x)\ dx < \infty
\end{equation}
we can define their \emph{Wasserstein distance} (also called \emph{Lip$'$-norm}) as
\begin{equation}\label{eq:lipdualdef}
W_1(u,v) := \sup_{\| \phi\|_{\Lip}\leq 1} \int_{\R} \phi(x)(u-v)(x)\ dx.
\end{equation}
\end{subequations}

The Wasserstein metric seems particularly suitable for comparing (approximate) solutions of conservation laws. Given an exact solution $u$ and an approximate solution $u_\Dx$ of \eqref{eq:cons_law}, the difference $u-u_\Dx$ has zero mass as long as the numerical scheme is conservative, and decays sufficiently fast under mild assumptions on the numerical scheme. Thus, the Wasserstein error $W_1(u,u_\Dx)$ will be well-defined and finite.
\begin{figure}
\centering
\subfigure[$W_1$-distance between two measures.]{\includegraphics[width=0.47\textwidth]{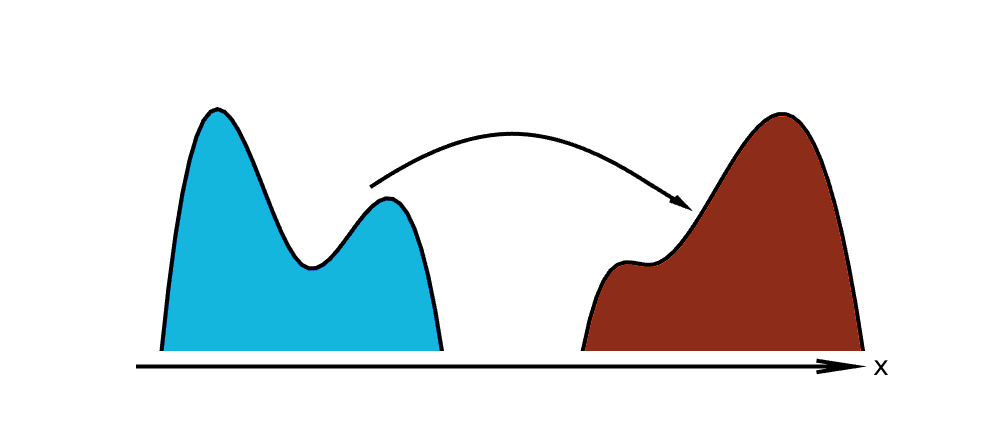}}
\subfigure[$W_1$-distance between exact and approximate solution of \eqref{eq:cons_law}.]{\includegraphics[width=0.47\textwidth]{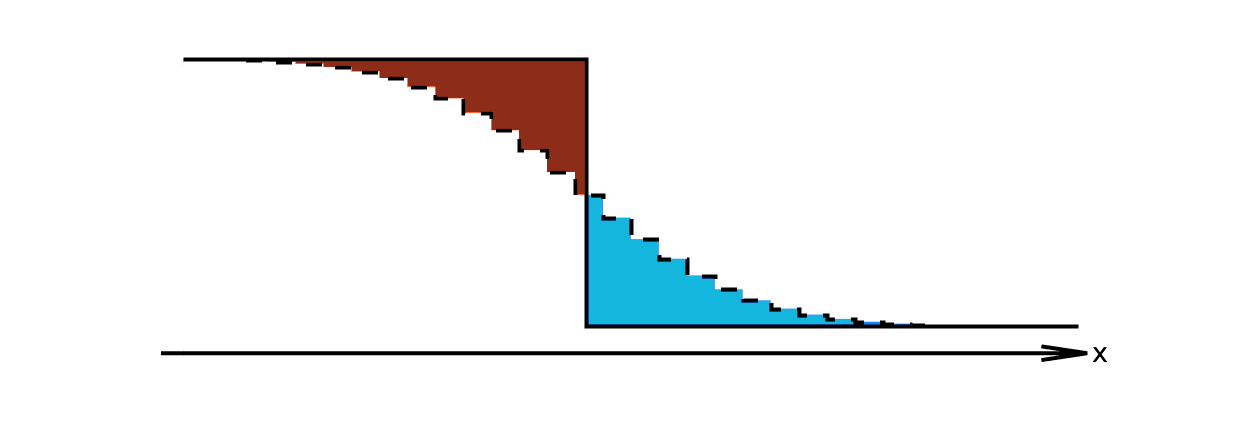}\hspace{4ex}}
\caption{The Wasserstein distance measures the amount of work (mass $\times$ distance) required to move mass from one place (light blue) to another (dark red).}
\label{fig:wasserstein}
\end{figure}

In the works of Tadmor et al.\ \cite{Tad91,nessy_conv92,nessy_conv}, the above metric \eqref{eq:lipdualdef} was studied extensively in the context of conservation laws, but under the different name of the \emph{$\Lip'$-norm}. They showed that a large class of monotone finite difference methods converge at a rate of $\Dx$ in the $\Lip'$-norm for \emph{any} initial data $u_0$. Thus, our result can be seen as an improvement over these earlier results for particular types of initial data, namely those whose solutions consist of shocks separated by constant states. 

In \cite{nessy_conv} it was argued that for initial data \emph{with a smooth solution}, the solution computed by a (formally) second-order TVD method converges at a rate of $\Dx^2$, as measured in the Wasserstein ($\Lip'$) norm. Although numerical experiments indicate that (formally) second-order TVD schemes \emph{always} converge at a rate of $\Dx^2$, no proof to this end is currently available. The present work can be seen as a step in this direction.

\subsection{Wasserstein and $L^1$ convergence}
Apart from convergence results in the $\Lip'$-norm, the only generic result on convergence rates of numerical schemes for conservation laws \eqref{eq:cons_law} is the $O(\Dx^\hf)$ rate in the $L^1$ norm, due to Kuznetsov \cite{Kuz76}. This was improved to $O(\Dx)$ by Teng and Zhang \cite{zhang_teng} for the particular case of piecewise constant entropy solutions. 

To motivate why the Wasserstein distance (or, equivalently, the $\Lip'$-norm) is more appropriate than the $L^1$ norm in the context of conservation laws, consider Figure \ref{fig:wasserstein}(b), which shows an exact solution of \eqref{eq:cons_law} containing a shock (solid curve), along with a typical numerical approximation (dashed curve). The $L^1$ distance $\|u_\Dx-u\|_{L^1(\R)}$ measures the area between the two graphs, indicated by dark red and light blue. The height of this area is $O(1)$ and the width $O(\Dx)$, so the $L^1$ norm will be $O(\Dx)$. The Wasserstein distance $W_1(u_\Dx,u)$, on the other hand, measures the amount of work (mass $\times$ distance) that goes into moving the surplus of mass in $u_\Dx$ (indicated in light blue) to behind the shock, where there is a shortage of mass (indicated in dark red). The area of mass that needs to be moved is $O(\Dx)$, and the distance between the blue and red areas is $O(\Dx)$, so the Wasserstein distance will be $O(\Dx\cdot \Dx) = O(\Dx^2)$.

Numerical evidence indicates that this difference between the $L^1$ and $W_1$ distances is \emph{generic}, in that any finite volume method will always be at most $O(\Dx)$ in $L^1$, and at most $O(\Dx^2)$ in $W_1$, in the presence of shocks. Intuitively we see this from the fact that even very high-order methods have a small amount of smearing near shocks, and therefore the approximate solution will be of the form indicated in (the somewhat exaggerated) Figure \ref{fig:wasserstein}(b). Note, however, that in smooth (but non-constant) regions of the solution, the accuracy of monotone methods degenerates to $O(\Dx)$---both in $L^1$ and $W_1$---as is to be expected.

\subsection{Background and outline of the paper}
The first monotone finite difference and finite volume methods for scalar conservation laws were developed by Lax, Godunov and others in the 1950s. The first generic result on convergence rates was the $O(\Dx^\hf)$ estimate in $L^1(\R)$ published by Kuznetsov in 1975--1976 \cite{Kuz76}. This approach was further developed in 1996 by Cockburn and Gremaud \cite{CG96}. A (pathological) counterexample due to \c Sabac (1997) shows that the $\Dx^\hf$ rate is sharp and cannot be improved without further assumptions on the initial data \cite{Sab97}. Numerical evidence indicates that the convergence rate is in fact higher---usually somewhere strictly between $\hf$ and 1---for more ``natural'', non-pathological initial data. This was confirmed in 1997 by Teng and Zhang~\cite{zhang_teng}, who proved $O(\Dx)$ convergence in $L^1(\R)$ in the particular case of piecewise constant data with only shocks, which is the setting considered in the present paper. Our approach follows that of Teng and Zhang, although differing in certain important aspects.

A different approach to obtaining convergence rates was initiated in Nessyahu and Tadmor's 1992 paper \cite{nessy_conv92}. Utilizing the dual equation studied by Tadmor in \cite{Tad91}, the authors proved that several finite volume schemes such as the Lax--Friedrichs, Engquist--Osher and Godunov schemes converge at a rate of $O(\Dx)$ in the $\Lip'$-norm, for arbitrary $\Lip^+$-bounded\footnote{A function $u_0:\R\to\R$ is $\Lip^+$-bounded if $\|u_0\|_{\Lip^+} := \sup_{x\neq y}\left(\frac{u_0(y)-u_0(x)}{y-x}\right)^+ < \infty$. Simply put, $u_0$ can have negative but not positive jump discontinuities.} initial data (see also \cite{nessy_conv}). Nessyahu and Tassa \cite{ness93} extended the results to  approximations $u^{\epsilon}$ with $\Lip^+$-unbounded initial data, which were shown to have a $\Lip'$-convergence rate of $O(\epsilon |\ln \epsilon|)$. As argued in the previous section, we believe that the $\Lip'$-norm (equivalently, the Wasserstein distance) plays an important role in 
 the context of numerics for conservation laws and deserves to be revisited.
%By standard interpolation arguments, this was extended to (lower-order) convergence rates in Sobolev spaces .

Next follows an outline of the present paper. In Section \ref{sec:setting} we describe the context of the Main Theorem and restate the theorem in more precise language. The rest of the paper is dedicated to the proof of the Main Theorem. We start by proving a $W_1$-stability estimate for monotone finite volume methods using a duality technique in Section \ref{sec:dual_stability}. In Section \ref{sec:trav_wave} we prove the Main Theorem in the case of a \emph{single} initial jump. In Section \ref{sec:multiple_jumps} we first prove the Main Theorem for an arbitrary number of shocks up until the first shock interaction time (Lemma \ref{lem:kshockerror}), and then in Section \ref{sec:after_interaction} we conclude the proof of the Main Theorem using induction on the number of shock interactions. In Section \ref{sec:numerical} we present numerical results to illustrate our main theorem. Finally, we end with some concluding remarks in Section \ref{sec:conclusion}.

We remark that although we only consider 3-point finite volume schemes, our proof works for any $(2p+1)$-point scheme. For notational convenience we only consider the former.

\section{Precise statement of main theorem}\label{sec:setting}
\subsection{The initial data and entropy solution}
We consider a decreasing and piecewise constant function $u_0:\R\to\R$, which can be written as
\begin{equation}\label{eq:initial_cond}
u_0(x) =\begin{cases}
    u^{(0)}       & \text{if } x<x^{1}\\
    u^{(k)}       & \text{if } x^{k}\leq x < x^{k+1}, \quad k=1,\dots,K-1 \\
    u^{(K)}       & \text{if } x^{K}\leq x,
  \end{cases}
\end{equation}
where $u^{(0)} > u^{(1)} > \dots >u^{(K)}$ and $x^1 < x^2 < \dots < x^K$. Here, $K$ is a finite number of jump discontinuities. If the flux function $f$ is convex then it is well-known that the entropy solution of \eqref{eq:cons_law} is
\begin{equation}\label{eq:acc_solution}
u(x,t) =\begin{cases}
    u^{(0)}       & \text{if } x<X^{1}(t) \\
    u^{(k)}       & \text{if } X^{k}(t) \leq x < X^{k+1}(t), \quad k=1,\dots,K-1 \\ 
    u^{(K)}       & \text{if } X^{K}(t)\leq x,
  \end{cases}
\end{equation}
where 
\begin{equation}\label{eq:characteristics}
X^{k}(t) = x^{k} + D^kt, \qquad D^k = \frac{f(u^{(k-1)})-f(u^{(k)})}{u^{(k-1)}-u^{(k)}}
\end{equation}
up to the first shock interaction time $t_{(1)}$,
\begin{align*}
t_{(1)} = \min_{2 \leq k \leq K} \frac{x^{k}-x^{k-1}}{D^{k-1}-D^k}.
\end{align*} 
After $t_{(1)}$ we can construct the solution $u(x,t)$ by finding the entropy solution to \eqref{eq:cons_law} with initial data $u(x,t_{(1)})$ with $K-1$ (or fewer) shocks and then continue this approach (at most) $K-2$ times. 

For simplicity we will denote the case of a single initial shock by
\begin{equation}\label{eq:heaviside}
H^{(k)}(x) = \begin{cases}
u^{(k-1)} & \text{if } x \leq 0 \\ 
u^{(k)} & \text{if } x > 0.
\end{cases} 
\end{equation}
Clearly, the entropy solution with initial data $H^{(k)}(x-x^k)$ is given by the \emph{traveling wave}
\begin{equation}\label{eq:travwave}
u(x,t) = H^{(k)}(x-X^k(t)).
\end{equation}

\subsection{Monotone schemes and discrete shocks}
We discretize the space-time domain $\R\times\R_+$ as $x_{i-\hf} = (i-\hf)\Dx$ and $t^n = n\Dt$ for $\Dx, \Dt>0$. We will denote $\lambda=\Dt /\Dx$. The exact solution $u$ of \eqref{eq:cons_law} is approximated in each cell $I_i=[x_{i-\hf},x_{i+\hf})$ by 
\[u_i^n \approx \frac{1}{\Dx}\int_{I_i} u(x,t^n) dx.
\]
The initial data is set as
$u_i^0 = \frac{1}{\Dx}\int_{I_i} u_0(x) dx.$ A numerical scheme of the form 
\[
u_i^{n+1} = G(u_{i-1}^n, u_i^n, u_{i+1}^n), \qquad i\in\Z
\]
is \emph{monotone} if $G$ is nondecreasing in all three arguments, and is \emph{conservative} if $\sum_i u_i^{n+1} = \sum_i u_i^n$ for all $n$. It is straightforward to show \cite[Proposition 1.1]{GR91} that the scheme is monotone if and only if there is a \emph{numerical flux function} $F(\cdot,\cdot)$ which is increasing in the first argument and decreasing in the second, such that
\begin{equation}\label{eq:num_met}
\begin{split}
u_i^{n+1} = u_i^{n}-\lambda \left(F(u_i^{n},u_{i+1}^{n}) - F(u_{i-1}^{n},u_{i}^{n})\right), \\
u_i^0 = \frac{1}{\Dx}\int_{I_i} u_0(x) dx,
\end{split}
\qquad\qquad (i\in\Z)
\end{equation}
and a certain CFL condition is satisfied. The scheme is \emph{consistent} if $G(u,u,u) = u$ for all $u\in\R$, or equivalently, if $F(u,u) = f(u)$ for all $u\in\R$. 

Given a numerical solution $u_i^n$ computed by \eqref{eq:num_met}, we extend it to all $(x,t) \in \R\times[0,\infty)$ by setting
\begin{equation}\label{eq:gridextension}
u_\Dx(x,t)=u_i^n \qquad \text{for } (x,t) \in I_i \times [t^n,t^{n+1}).
\end{equation}
It is clear that $u_\Dx$ satisfies for all $(x,t)\in\R\times\R_+$
\begin{equation}\label{eq:nummetphys}
u_\Dx(x,t+\Dt) =  u_\Dx(x,t) -\lambda\Bigl[F\bigl(u_\Dx(x,t),\ u_\Dx(x+\Dx,t)\bigr) - F\bigl(u_\Dx(x-\Dx,t),\ u_\Dx(x,t)\bigr)\Bigr],
\end{equation}
or written in terms of the normalized function $U(\xi,\eta) := u_\Dx(\xi\Dx, \eta\Dt)$,
\[
U(\xi,\eta+1) =  U(\xi,\eta) -\lambda\Bigl[F\bigl(U(\xi,\eta),\ U(\xi+1,\eta)\bigr) - F\bigl(U(\xi-1,\eta),\ U(\xi,\eta)\bigr)\Bigr].
\]

Analogous to the traveling wave solution \eqref{eq:travwave}, we ask whether there exist numerical solutions satisfying $u_\Dx(x,t+\Dt) = u_\Dx(x-D^k\Dt,t)$ for all $(x,t)$, or equivalently, $U(\xi,\eta+1) = U(\xi-D^k\lambda,\eta)$ for all $(\xi,\eta)$. 
\begin{definition}
A \define{discrete shock} for \eqref{eq:num_met} connecting $u^{(k-1)}$ and $u^{(k)}$ is a function $U^{(k)}:\R\to\R$ satisfying
\begin{equation}\label{eq:trav_wave_eq}
U^{(k)}(\xi-D^k\lambda) =  U^{(k)}(\xi) -\lambda \Bigl[F\bigl(U^{(k)}(\xi),U^{(k)}(\xi+1)\bigr) - F\bigl(U^{(k)}(\xi-1),U^{(k)}(\xi)\bigr)\Bigr] \quad \forall\ \xi\in\R
\end{equation}
and $\displaystyle\lim_{\xi\to-\infty} U^{(k)}(\xi) = u^{(k-1)}$, $\displaystyle\lim_{\xi\to+\infty} U^{(k)}(\xi) = u^{(k)}$. 
\end{definition}

Note that \eqref{eq:trav_wave_eq} does not depend on $\Dx$ or $\Dt$, only on their ratio $\lambda = \frac{\Dt}{\Dx}$. Thus, any statement about discrete shocks (such as the existence Theorem \ref{thm:disctravwave}) will be independent of the mesh size.
The existence of discrete shocks for \eqref{eq:num_met}---essential in the proof of our main result---has been proven in several important cases. The first existence result was given by Jennings \cite{jennings_74}, who proved the existence of a discrete shock for \eqref{eq:num_met} provided $D^k\lambda$ is rational, that the numerical flux $F$ is differentiable and the scheme is strictly monotone (i.e., $F$ is \emph{strictly} increasing/decreasing in the first/second argument). Enquist and Osher \cite{enquist81} showed existence for general monotone schemes with a differentiable flux. The existence of discrete shocks, for both $D^k\lambda$ rational and irrational, has been proven for the Godunov scheme by Fan \cite{fan_98}. Serre \cite{serre_04, serre_07} proved existence of discrete shocks for (both strict and non-strict) monotone schemes, including the Godunov scheme, for $D^k\lambda$ irrational as well. Assuming $D^k\lambda$ rational, Fan \cite{fan00} established existence of discrete shocks for, in addition to the schemes mentioned above, second-order MUSCL schemes.

We summarize these results as follows:
\begin{theorem}\label{thm:disctravwave}
Let $u^{(k)}<u^{(k-1)}$ and $f$ in \eqref{eq:cons_law} be convex. 
Assume that the conservative finite volume method \eqref{eq:num_met} is monotone and that $F$ is Lipschitz continuous in both arguments. Then for every value $u^*\in\bigl(u^{(k)}, u^{(k-1)}\bigr)$ there exists a unique Lipschitz continuous discrete shock for \eqref{eq:num_met} connecting $u^{(k-1)}$ and $u^{(k)}$, satisfying
\begin{equation}
U^{(k)}(0) = u^*,
\end{equation}
and the point values $U^{(k)}(\xi)$ depend continuously on the choice of $u^*$. Moreover, there are constants $\alpha_k, \beta_k > 0$ such that
\begin{subequations}\label{eq:exponentialdecay}
\begin{align}
\Bigl|U^{(k)}(\xi) - u^{(k-1)}\Bigr| \leq \beta_k e^{-\alpha_k|\xi|}& &&\forall\ \xi \leq 2, \label{eq:exponentialdecayA} \\
\Bigl|U^{(k)}(\xi) - u^{(k)}\Bigr| \leq \beta_k e^{-\alpha_k|\xi|}& &&\forall\ \xi \geq -2, \label{eq:exponentialdecayB}
\end{align}
\end{subequations}
and furthermore,
\begin{equation}\label{eq:jump_condition}
\sum_{i\in\Z} \left(U^{(k)}(i+\zeta) - U^{(k)}(i) \right) = \zeta \bigl(u^{(k-1)}-u^{(k)}\bigr) \qquad \forall \ \zeta \in \R.
\end{equation}
\end{theorem}
We refer to \cite{jennings_74,enquist81,fan_98,serre_04,serre_07} for the proof.

%The following is a simple corollary, proven using the intermediate value theorem on $u^*$.
%\begin{corollary}
%Under the same assumptions as in Theorem \ref{thm:disctravwave}, there exists a Lipschitz continuous discrete shock of \eqref{eq:num_met} connecting $u^{(k-1)}$ and $u^{(k)}$ satisfying
%\begin{equation}\label{eq:zeromassfirst}
%\int_{\R} H^{(k)}(x) - V(x) \ dx = 0.
%\end{equation}
%\end{corollary}
\subsection{The discrete Wasserstein distance and $W_1$-contractivity}
Before stating the main theorem, we need to define $W_1$-contractivity and a discrete version of the Wasserstein distance. 
\begin{definition}
For functions with piecewise constant values $u_\Dx(x)=\sum_i u_i\ind_{I_i}(x)$ and $v_\Dx(x)=\sum_i v_i\ind_{I_i}(x)$ satisfying
\begin{subequations}\label{eq:dwasserdef}
\begin{equation}\label{eq:dwassercondition}
\sum_{i\in\Z} u_i - v_i = 0, \qquad \sum_{i\in\Z} |i| |u_i-v_i| < \infty,
\end{equation}
we define their \emph{discrete Wasserstein distance} (or \emph{$\DLip'$-distance}) as
\begin{equation}\label{eq:dlipdualdef}
W_{1,D}(u_\Dx,v_\Dx):=\sup_{\|\phi\|_\DLip\leq1} \sum_i \phi_i (u_i-v_i)\Dx.
\end{equation}
\end{subequations}
\end{definition}
Note that the condition \eqref{eq:dwassercondition} is equivalent to \eqref{eq:wassercondition} for piecewise constant functions. The supremum in \eqref{eq:dlipdualdef} is taken over all grid functions $\phi(x)=\sum_i \phi_i\ind_{I_i}(x)$ with
\[
\|\phi\|_\DLip := \sup_{i\in\Z} \left|\frac{\phi_{i+1}-\phi_i}{\Dx}\right| \leq 1.
\]
\begin{definition}
We say that the monotone scheme \eqref{eq:num_met} is \emph{$W_1$-contractive} if the following holds.
Let $u_\Dx(x,t)$ and $v_{\Dx}(x,t)$ be computed by the inhomogeneous schemes
\begin{equation}\label{eq:inhom_num}
\begin{cases}
\begin{split}
u_i^{n+1}&= u_i^{n}-\lambda \Bigl(F(u_i^{n},u_{i+1}^{n}) - F(u_{i-1}^{n},u_{i}^{n})\Bigr) + h_i^n\Dt, &\quad u_i^0 = \frac{1}{\Dx}\int_{I_i} u_0(x)\ dx, \\
v_i^{n+1}&=v_i^{n}-\lambda \Bigl(F(v_i^{n},v_{i+1}^{n}) - F(v_{i-1}^{n},v_{i}^{n})\Bigr) + g_i^n\Dt, &\quad v_i^0 = \frac{1}{\Dx}\int_{I_i} v_0(x)\ dx,
\end{split}
\end{cases}
\end{equation}
where $u_0$ and $v_0$ are decreasing and $h$ and $g$ are such that the difference $h^n-g^n$ satisfies \eqref{eq:dwassercondition} for every $n$. Then
\begin{equation}\label{eq:wass_err_nonhom}
\begin{split}
W_1\bigl(u_\Dx(t^N), v_{\Dx}(t^N)\bigr) &\leq W_{1,D}\bigl(u_\Dx(0), v_\Dx(0)\bigr) + \Dt \sum_{n=0}^{N-1} W_{1,D}\bigl(h^n, g^n\bigr) 
%&\leq W_1\bigl(u_0, v_0\bigr) + \Dt \sum_{n=0}^{N-1} W_{1,D}\bigl(h^n, g^n\bigr).
\end{split}
\end{equation}
for all $N\in\N$.
\end{definition}

\subsection{Statement of the main theorem}\label{sec:mainthm}
Using the existence of discrete shocks, we will prove the following result.

\begin{theorem}[Main result]\label{thm:main}
Let $f$ be convex and let $u_0$ be of the form \eqref{eq:initial_cond}. Let $u$ be the exact solution \eqref{eq:acc_solution} and let $u_\Dx$ be the numerical solution computed by the $W_1$-contractive, monotone scheme \eqref{eq:num_met}. Then there is a constant $C>0$ depending on $K$ and the size of the jumps (but not on $\Dx$) such that
\begin{equation}\label{eq:main_conv_res}
W_1\bigl(u(\cdot,t^n),u_\Dx(\cdot,t^n)\bigr) \leq C\Dx^2 \qquad \forall\ n\in\N.
\end{equation}
\end{theorem}

%\begin{remark}
%We can only consider the error for the discrete times $t^n$, as we in order to use the $W_1$-distance, need the difference $u(x,t)-u_\Dx(x,t)$ to have zero mass. If we have zero mass in one timestep $t^n$, we will not necessarily have zero mass for $t$ strictly between $t^n$ and $t^{n+1}$, as $u_\Dx(x,t)=u_\Dx(x,t^n)$, but $u(x,t)$ will have changed.\footnote{Remove this?}
%\end{remark}

\section{Discrete dual problem and $W_1$-contractivity}\label{sec:dual_stability}
In this section we prove $W_1$-contractivity for a class of monotone schemes using a dual argument (Sections \ref{sec:dualproblem} and \ref{sec:W1contractive}). We begin by describing the relationship between the Wasserstein and discrete Wasserstein distances $W_1$ and $W_{1,D}$.

\subsection{The $W_1$ and $W_{1,D}$ distances}
In the present setting of one spatial dimension, both the Wasserstein and the discrete Wasserstein distances admit a particularly simple form:
\begin{equation}\label{eq:wassalternative}
W_1(u,v) = \int_{\R}\left| \int_{-\infty}^x (u-v)(y)\ dy\right| dx
\end{equation}
and
\begin{equation}\label{eq:dwassalternative}
W_{1,D}(u_\Dx,v_\Dx) = \sum_{i\in\Z}\left| \sum_{j<i} u_j-v_j \right| \Dx^2.
\end{equation}
These are obtained by integrating \eqref{eq:lipdualdef} by parts and replacing $\frac{d\phi}{dx}$ by its maximum value of 1 (respectively, summation by parts of \eqref{eq:dlipdualdef} and replacing $\frac{\phi_{i+1}-\phi_i}{\Dx}$ by its maximum value of 1). Using these formulas, it is easy to show that the Wasserstein and discrete Wasserstein distances coincide in some important cases:

\begin{proposition}\label{prop:pos_err}
Let $u_\Dx$ and $v_\Dx$ be piecewise constant functions satisfying \eqref{eq:dwassercondition}. Then
\begin{equation}\label{eq:W1vsW1D}
W_{1}(u_\Dx,v_\Dx) \leq W_{1,D}(u_\Dx,v_\Dx).
\end{equation}
If additionally
\begin{align}\label{eq:w1_assump}
\int_{-\infty}^x (u_\Dx-v_\Dx)(y)\ dy \geq 0 \quad (\text{or } \leq 0) \qquad \forall\ x\in\R
\end{align}
then 
\begin{equation}\label{eq:equal_disc_con}
W_1(u_\Dx,v_\Dx) = W_{1,D}(u_\Dx,v_\Dx). 
\end{equation}
\end{proposition}
\begin{proof}
Denote $\phi_i = \frac{1}{\Dx} \int_{I_i} \phi(x) dx$ for any $\phi\in\Lip(\R)$. Then $\|\phi\|_\DLip \leq \| \phi\|_{\Lip}$, so
\begin{align*}
W_{1}(u_\Dx,v_\Dx) &=  \sup_{\| \phi\|_{\Lip}\leq 1} \int_\R \phi(x)\left(u_\Dx-v_\Dx\right)(x)\ dx \\
&= \sup_{\| \phi\|_{\Lip}\leq 1} \sum_i \int_{I_i} \phi(x) \ dx \ (u_i-v_i) \\
&= \sup_{\| \phi\|_{\Lip}\leq 1} \sum_i \phi_i (u_i-v_i)\Dx \leq W_{1,D}(u_\Dx,v_\Dx).
\end{align*}

Under the condition \eqref{eq:w1_assump}, the terms inside the absolute values in \eqref{eq:wassalternative} and \eqref{eq:dwassalternative} have a fixed sign, so the absolute value can be moved outside the integral (sum). By using the fact that $u_\Dx$, $v_\Dx$ are piecewise constant, we get
\begin{align*}
W_1(u_\Dx,v_\Dx) &= \biggl|\int_{\R} \int_{-\infty}^x (u_\Dx(y)-v_\Dx(y))\ dydx\biggr|
= \biggl|\sum_{i\in\Z}\int_{I_i} \int_{-\infty}^x (u_\Dx(y)-v_\Dx(y))\ dydx\biggr| \\
&= \biggl|\sum_{i\in\Z}\Biggl(\sum_{j<i}(u_j-v_j)\Dx^2 + (u_i-v_i)\int_{I_i}(x-x_{i-\hf})\ dx\Biggr)\biggr| \\
&= \biggl| \sum_{i\in\Z}\Biggl(\sum_{j<i}(u_j-v_j)\Dx^2 + (u_i-v_i)\frac{\Dx^2}{2}\Biggr) \biggr| \\
&= \biggl| \sum_{i\in\Z} \sum_{j<i}(u_j-v_j)\Dx^2 \biggr| = W_{1,D}(u_\Dx,v_\Dx)
\end{align*}
(the second-last step following from the fact that $\sum_i u_i-v_i = 0$).
\end{proof}

Next, we prove a simple result on the Wasserstein error of projection onto piecewise constant functions. Below we use the standard notation
\[
{\rm TV}(v) = \limsup_{h\to0}\int_\R \left|\frac{v(x+h)-v(x)}{h}\right|\ dx
\]
and we let ${\rm BV}(\R) = \{ v:\R\to\R\ :\ {\rm TV}(v)<\infty\}$.

\begin{proposition}
Let $v\in{\rm BV}(\R)$ and define $v_i = \frac{1}{\Dx}\int_{I_i}v(x) dx$ and $v_\Dx(x) = \sum_{i\in\Z} v_i\ind_{I_i}(x)$, the piecewise constant projection of $v$. Then
\begin{equation}\label{eq:project_err}
W_1\bigl(v, v_\Dx\bigr) \leq {\rm TV}(v) \Dx^2. 
\end{equation}
\end{proposition}
\begin{proof}
We use the formula \eqref{eq:wassalternative}:
\begin{align*}
W_1\bigl(v, v_\Dx\bigr) &= \int_{\R}\left| \int_{-\infty}^x v(y) - v_\Dx(y)\ dy\right| dx \\
&= \sum_i \int_{I_i}\left| \int_{-\infty}^x v(y) - v_\Dx(y)\ dy\right| dx \\
&= \sum_i \int_{I_i}\left| \int_{x_{i-\hf}}^x v(y) - v_\Dx(y)\ dy \right| dx \\
&\leq \Dx \sum_i \int_{I_i} \left| v(y) - v_\Dx(y)\right| \ dy \leq \Dx^2 {\rm TV}(v).\qedhere
\end{align*}
\end{proof}

\subsection{The discrete dual problem}\label{sec:dualproblem}
The following auxiliary lemma gives the stability of a (backwards) dual equation, and can be seen as a (special case of a) discrete version of \cite[Theorem 2.2]{Tad91}.
\begin{lemma}\label{lem:dualstability}
Let $\phi_i^n$ satisfy the backward difference equation
\begin{align}\label{eq:backwardprob}
\frac{\phi_i^{n+1}-\phi_i^{n}}{\Dt} + \frac{1}{\Dx}\Bigl(A_i^n\left( \phi_{i+1}^{n+1}-\phi_i^{n+1}\right) +B_i^n \left( \phi_i^{n+1}-\phi_{i-1}^{n+1}\right)\Bigr)=0
\end{align}
where
\begin{equation}\label{eq:backwardcond}
0\leq A_i^n\leq A_{i-1}^n, \qquad 0\geq B_{i-1}^n \geq B_i^n, \qquad \lambda\left(A_{i-1}^n - B_i^n\right) \leq 1 \qquad \text{for all $i$ and $n$.}
\end{equation}
Then
\[
\|\phi^n\|_\DLip \leq \|\phi^{n+1}\|_\DLip \qquad \forall\ n.
\]
\end{lemma}
\begin{proof}
Define $\psi_i^n := (\phi_i^n-\phi_{i-1}^n)/\Dx$. Taking a difference of \eqref{eq:backwardprob} in space, we see that $\psi$ satisfies
\begin{align*}
\psi_{i}^{n} &= \psi_{i}^{n+1}  + \frac{\Dt}{\Dx}\Bigl(A_i^n\psi_{i+1}^{n+1} - A_{i-1}^n\psi_{i}^{n+1} +B_i^n\psi_{i}^{n+1} - B_{i-1}^{n}\psi_{i-1}^{n+1} \Bigr) \\
&= \left(1-\lambda\left(A_{i-1}^n-B_i^n\right)\right)\psi_i^{n+1}  + \lambda A_i^n\psi_{i+1}^{n+1} + \bigl(-\lambda B_{i-1}^n\bigr)\psi_{i-1}^{n+1}.
\end{align*}
By our assumptions on $\{A_i^n\}_{i \in \Z}$ and $\{B_i^n\}_{i \in \Z}$, each coefficient is nonnegative, so upon taking absolute values we get
\[
|\psi_{i}^{n}| \leq \left(1-\lambda\left(A_{i-1}^n-B_i^n\right)\right)|\psi_i^{n+1}|  + \lambda A_i^n|\psi_{i+1}^{n+1}| -\lambda B_{i-1}^n|\psi_{i-1}^{n+1}|.
\]
Since $\{A_i^n\}_{i \in \Z}$ and $\{B_i^n\}_{i \in \Z}$ are decreasing sequences, the coefficients sum up to at most 1, and so
\[
|\psi_{i}^{n}| \leq \max\left(|\psi_i^{n+1}|,\ |\psi_{i+1}^{n+1}|,\ |\psi_{i-1}^{n+1}|\right).
\]
This proves our claim.
\end{proof}

We use Lemma \ref{lem:dualstability} to prove the following $W_1$-contractivity result using a duality argument.

\begin{theorem}[$W_1$-contractivity]\label{thm:stability}
Let $u_0, v_0 \in L^\infty(\R)$ be decreasing functions whose difference $u_0-v_0$ satisfies \eqref{eq:wassercondition}. Let $u_\Dx(x,t)$ and $v_{\Dx}(x,t)$ be computed by the monotone, inhomogeneous schemes \eqref{eq:inhom_num}. Assume that we can write
\begin{equation}\label{eq:lip_flux}
F(u_{i}^{n}, u_{i+1}^{n})-F(v_{i}^{n}, v_{i+1}^{n})=A_i^n(u_i^n-v_i^n)+B_{i+1}^n(u_{i+1}^n-v_{i+1}^n),
\end{equation}
where $A_i^n$ and $B_{i+1}^n$ satisfy \eqref{eq:backwardcond}. Further, assume that the CFL condition 
\begin{align}\label{eq:cfl}
\frac{\Dt}{\Dx}\max_{\underline{u} \leq a,b \leq \overline{u}} \left| \frac{\partial F}{\partial a}(a,b) \right|\leq \frac{1}{2}, \qquad 
\frac{\Dt}{\Dx}\max_{\underline{u} \leq a,b \leq \overline{u}} \left| \frac{\partial F}{\partial b}(a,b) \right| \leq \frac{1}{2},
\end{align}
where $\underline{u}=\min_{i} u_i^0$ and $\overline{u}=\max_{i}u_i^0$, is satisfied. Then the scheme is $W_1$-contractive.
\end{theorem}
If the flux is only Lipschitz continuous, we interpret $\partial_a F$ and $\partial_b F$ in \eqref{eq:cfl} as the Lipschitz constants of $F$.

\begin{proof}
We follow a duality approach similar to \cite[Theorem 2.1]{nessy_conv92}, although our proof differs in some important aspects. By \eqref{eq:W1vsW1D} in Proposition \ref{prop:pos_err}, it suffices to show that
\begin{equation*}
W_{1,D}\bigl(u_\Dx(t^N), v_{\Dx}(t^N)\bigr) \leq W_{1,D}\bigl(u_\Dx(0), v_\Dx(0)\bigr) + \Dt \sum_{n=0}^{N-1} W_{1,D}\bigl(h^n, g^n\bigr).
\end{equation*}
Let $\phi^{n+1}$ be any grid function with $\|\phi^{n+1}\|_\DLip\leq1$. Then
\begin{align*}
&\Dx \sum_{i} \phi_i^{n+1}\bigl(u_i^{n+1}-v_i^{n+1}\bigr) \\*
=\ &\Dx \sum_{i} \phi_i^{n+1}\Bigl[u_i^{n}-v_i^{n} - \lambda\Bigl( F(u_i^{n}, u_{i+1}^{n})-F(u_{i-1}^{n}, u_{i}^{n})-F(v_i^{n}, v_{i+1}^{n}) + F(v_{i-1}^{n}, v_{i}^{n})\Bigr) + \Dt\bigl(h_i^n-g_i^n\bigr)\Bigr].
\end{align*}
Using \eqref{eq:lip_flux}, we get
\begin{align*}
&\Dx \sum_{i} \phi_i^{n+1}\Bigl[u_i^{n}-v_i^{n} - \lambda\Bigl(A_i^n\bigl(u_i^n-v_i^n\bigr)+B_{i+1}^n\bigl(u_{i+1}^n-v_{i+1}^n\bigr) -A_{i-1}^n\bigl(u_{i-1}^n-v_{i-1}^n\bigr)-B_i^n\bigl(u_{i}^n-v_{i}^n\bigr)\Bigr) \\
&\qquad\qquad\qquad\quad + \Dt\bigl(h_i^n-g_i^n\bigr)\Bigr] \\
&=  \Dx \sum_i\Bigl[\phi_i^{n+1} + \lambda\Bigl(A_i^n\bigl(\phi_{i+1}^{n+1}-\phi_{i}^{n+1}\bigr) + B_i^n\bigl(\phi_{i}^{n+1}-\phi_{i-1}^{n+1}\bigr)\Bigr) \Bigr](u_i^n-v_i^n) + \sum_i \phi_i^{n+1}\bigl(h_i^n-g_i^n\bigr)\Dx\Dt,
\end{align*}
where we in the last step have applied summation by parts and used the fact that
\[
\lim_{i \to \pm \infty} \phi_{i+1}^{n+1}A_{i}^n\bigl(u_{i}^n-v_{i}^n\bigr) = 0, \qquad \lim_{i \to \pm \infty} \phi_{i}^{n+1}B_{i}^n\bigl(u_{i}^n-v_{i}^n\bigr) = 0, 
\]
(which follows from the fact that $u^0-v^0$ satisfies \eqref{eq:dwassercondition}).
We choose now $\phi_i^n$ to satisfy the discrete backward problem
\begin{align}\label{eq:backwardprob2}
\phi_i^{n} = \phi_i^{n+1} + \lambda\Bigl(A_i^n\left( \phi_{i+1}^{n+1}-\phi_i^{n+1}\right) +B_i^n \left( \phi_i^{n+1}-\phi_{i-1}^{n+1}\right) \Bigr), \qquad n=0,\dots,N-1
\end{align}
(which coincides with \eqref{eq:backwardprob}). Then
\begin{align}\label{eq:sum_equality}
\sum_{i} \phi_i^{n+1}\bigl(u_i^{n+1}-v_i^{n+1}\bigr) \Dx = \sum_{i} \phi_i^{n}\bigl(u_i^{n}-v_i^{n} \bigr)\Dx + \sum_i \phi_i^{n+1}\bigl(h_i^n-g_i^n\bigr)\Dx\Dt.
\end{align}
Iterating over $n$, we find that
\[
\sum_{i} \phi_i^{N}\bigl(u_i^{N}-v_i^{N}\bigr) \Dx = \sum_{i} \phi_i^0\bigl(u_i^0-v_i^0 \bigr)\Dx + \sum_{n=0}^{N-1}\sum_i \phi_i^{n+1}\bigl(h_i^n-g_i^n\bigr)\Dx\Dt. 
\]
Under our assumptions on $\lambda$, $u_0, v_0$, $\{A_i^n\}_{i \in \Z}$, $\{B_i^n\}_{i \in \Z}$ and the CFL number $\lambda=\frac{\Dt}{\Dx}$, the conditions \eqref{eq:backwardcond} of Lemma \ref{lem:dualstability} are satisfied. Thus, $\|\phi^n\|_\DLip \leq \|\phi^N\|_\DLip$ for $n=0,\dots,N$, and so
\begin{align*}
\sum_{i} \phi_i^{N}\bigl(u_i^{N}-v_i^{N}\bigr) \Dx &\leq \|\phi^0\|_\DLip W_{1,D}\bigl(u_\Dx(0), v_\Dx(0)\bigr) + \sum_{n=0}^{N-1} \|\phi^{n+1}\|_\DLip W_{1,D}\bigl(h^n, g^n\bigr)\Dt \\
&\leq \|\phi^N\|_\DLip \left(W_{1,D}\bigl(u_\Dx(0), v_\Dx(0)\bigr) + \sum_{n=0}^{N-1}W_{1,D}\bigl(h^n, g^n\bigr)\Dt\right)
%&\leq \|\phi^N\|_\DLip \left(W_{1}\bigl(u_0, v_0\bigr) + \sum_{n=0}^{N-1}W_{1,D}\bigl(h^n, g^n\bigr)\Dt\right).
\end{align*}
Taking the supremum over all $\phi^N$ with $\|\phi^N\|_\DLip \leq 1$ gives the desired result.
\end{proof}

\subsection{$W_1$-contractive schemes}\label{sec:W1contractive}
For decreasing initial data $u_0, v_0 \in L^\infty(\R)$, several monotone schemes satisfy condition \eqref{eq:backwardcond} and are therefore $W_1$-contractive. In this section, we show that the condition is satisfied for the Lax--Friedrichs, Enquist--Osher and Godunov schemes. These schemes have Lipschitz continuous fluxes and so we can write
\begin{align}
F(u_{i}^{n}, u_{i+1}^{n})-F(v_{i}^{n}, v_{i+1}^{n}) &= \int_0^1 \frac{\partial}{\partial s} \left( F(v_{i}^{n} + s(u_{i}^{n}-v_{i}^{n}), v_{i+1}^{n}+s(u_{i+1}^{n}-v_{i+1}^{n}) \right)ds \nonumber \\
 &= \int_0^1 \frac{\partial F}{\partial a}\left((1-s)v_{i}^{n} + su_{i}^n, (1-s)v_{i+1}^{n} + su_{i+1}^n \right)  ds (u_{i}^{n}-v_{i}^{n}) \nonumber \\
&\quad + \int_0^1 \frac{\partial F}{\partial b}\left((1-s)v_{i}^{n} + su_{i}^n, (1-s)v_{i+1}^{n} + su_{i+1}^n \right)  ds (u_{i+1}^{n}-v_{i+1}^{n}) \nonumber \\
 &= A_i^n (u_{i}^{n}-v_{i}^{n}) + B_{i+1}^n (u_{i+1}^{n}-v_{i+1}^{n}). \label{eq:gen_flux_expr}
\end{align}
If $u_{i}^0 \leq u_{i-1}^0$ and $v_{i}^0 \leq v_{i-1}^0$ for all $i$, then by monotonicity, also $u_{i}^{n} \leq u_{i-1}^{n}$ and $v_{i}^{n} \leq v_{i-1}^{n}$ for all $i \in \Z$ and $n\in\N$. Hence, the convex combinations above satisfy $(1-s)v_{i}^{n} + su_{i}^n \leq (1-s)v_{i-1}^{n} + su_{i-1}^n$ for all $i \in \Z$ and $s \in [0,1]$.

\subsubsection{Lax--Friedrichs scheme}
The Lax--Friedrichs flux,
\begin{equation*}
F(a,b) = \frac{1}{2} \bigl( f(a) + f(b) \bigr) -\frac{1}{2\lambda}\left(b-a\right)
\end{equation*}
is differentiable, and the partial derivatives are
\begin{align*}
\frac{\partial F}{\partial a}= \frac{1}{2} \Bigl( f'(a)+\frac{1}{\lambda} \Bigr), \qquad\qquad \frac{\partial F}{\partial b}= \frac{1}{2} \Bigl( f'(b)-\frac{1}{\lambda} \Bigr).
\end{align*}
Inserting these into \eqref{eq:gen_flux_expr} and remembering that the flux $f$ is assumed to be convex and that $|f'(u)|\leq \lambda$, we see that \eqref{eq:backwardcond} holds for this scheme.

\subsubsection{Enquist--Osher scheme}
Similarily, the Enquist--Osher flux,
\begin{equation*}
F(a,b) = \frac{1}{2} \bigl( f(a) + f(b) \bigr) -\frac{1}{2}\int_a^b |f'(\alpha)|d\alpha
\end{equation*} 
is differentiable with partial derivatives
\begin{align*}
\frac{\partial F}{\partial a}= \frac{1}{2}\bigl(f'(a) + |f'(a)| \bigr), \qquad\qquad \frac{\partial F}{\partial b}= \frac{1}{2}\bigl(f'(b) - |f'(b)| \bigr).
\end{align*}
As for the Lax--Friedrichs flux, it is easily seen that the scheme satisfies \eqref{eq:backwardcond}.

\subsubsection{Godunov scheme}
As $f$ is convex and $b \leq a$ for decreasing solutions, the Godunov flux takes the simple form
\begin{equation}\label{eq:godflux}
F(a,b) = \max_{b\leq u \leq a} f(u) = 
\begin{cases}
f(a) & \text{if } f(a) \geq f(b) \\
f(b) & \text{if } f(a) < f(b)
\end{cases}
\end{equation}
with
\begin{align*}
\frac{\partial F}{\partial a}= 
\begin{cases}
f'(a) & \text{if } f(a) \geq f(b) \\
0 & \text{if } f(a) < f(b)
\end{cases} \qquad\qquad 
\frac{\partial F}{\partial b}=
\begin{cases}
0 & \text{if } f(a) \geq f(b) \\
f'(b) &\text{if } f(a) < f(b).
\end{cases}
\end{align*}
With $f$ convex, $f'$ is increasing, and we can observe, by inserting $\partial F_a$ and $\partial F_b$ into \eqref{eq:gen_flux_expr}, that the Godunov scheme also fulfills \eqref{eq:backwardcond}.

\section{Convergence rate for one initial jump}\label{sec:trav_wave}
In this section we consider the case of a single initial jump $u_0(x) = H^{(k)}(x-x^k)$ (cf.\ \eqref{eq:heaviside}) with $u^{(k-1)}>u^{(k)}$, and we prove that any $W_1$-contractive, monotone finite volume scheme converges at a rate of $\Dx^2$ (Theorem \ref{thm:singleshockconv}). To this end, we first show that the difference between the exact solution and a discrete shock is $O(\Dx^2)$ (Lemma \ref{lem:trav_wave_err}), and then apply Theorem \ref{thm:stability} and the triangle inequality to conclude.

Recall from \eqref{eq:travwave} that the entropy solution in this case is
\begin{equation}\label{eq:singlewaveentrsoln}
H^{(k)}(x,t) := H^{(k)}\bigl(x-X^k(t)\bigr).
\end{equation}
According to Theorem \ref{thm:disctravwave}, there exists for every $u^* \in (u^{(k)}, u^{(k-1)})$ a discrete shock $U^{(k)}$ connecting $u^{(k-1)}$ and $u^{(k)}$ such that $U^{(k)}(0)=u^*$. Define
\begin{align}\label{eq:const_TW_initial}
V^{(k)}(x)= \sum_{i\in\Z} U^{(k)}(i)\ind_{I_i}(x).
\end{align}
By selecting the middle state $u^*$ appropriately, we may assume that 
\begin{equation}\label{eq:zeromassfirst}
\int_\R H^{(k)}(x) - V^{(k)}(x)\ dx = 0.
\end{equation}
Together with the exponential decay property \eqref{eq:exponentialdecay}, this ensures that $H^{(k)}-V^{(k)}$ satisfies \eqref{eq:wassercondition}, so that the Wasserstein distance $W_1\bigl(H^{(k)}, V^{(k)}\bigr)$ is well-defined. Defining
\begin{equation}\label{eq:const_TW}
V^{(k)}(x,t)=\sum_{i} U^{(k)}\left(\frac{x_i-X^k(t)}{\Dx}\right)\ind_{I_i}(x),
\end{equation}
it is clear that $V^{(k)}$ is a solution of the numerical scheme \eqref{eq:nummetphys} with initial data $V^{(k)}(x,0)=V^{(k)}(x)$, and satisfies the exponential decay
%\begin{equation}\label{eq:exponentialdecay2}
%\bigl|V^{(k)}(x,t) - H^{(k)}(x,t)\bigr| \leq \beta_k e^{-\alpha_k|x|/\Dx}
%\end{equation}
\begin{subequations}\label{eq:exponentialdecay2}
\begin{align}
\Bigl|V^{(k)}(x,t) - u^{(k-1)}\Bigr| \leq \beta_k e^{-\alpha_k|x-X^k(t)|/\Dx}& &&\forall\ x\leq X^k(t)+2\Dx, \label{eq:exponentialdecay2A} \\
\Bigl|V^{(k)}(x,t) - u^{(k)}\Bigr| \leq \beta_k e^{-\alpha_k|x-X^k(t)|/\Dx}& &&\forall\ x \geq X^k(t)-2\Dx \label{eq:exponentialdecay2B}
\end{align}
\end{subequations}
(possibly with new constants $\alpha_k$ and $\beta_k$). Moreover, from the property \eqref{eq:jump_condition} we find that
\begin{align*}
0 =&\ \int_{\R} V^{(k)}(x) - H^{(k)}(x) \ dx = \Dx \sum_i \left[ U^{(k)}(i) - \frac{1}{\Dx} \int_{I_i} H^{(k)}(x) dx \right] \\
=&\ \Dx\sum_i\left[U^{(k)}\left(\frac{x_i-X^k(t)}{\Dx}\right) -\frac{1}{\Dx} \int_{I_i} H^{(k)}(x-X^k(t)) dx \right] \\
&+ \Dx \sum_i \left[ U^{(k)}\left(\frac{x_i}{\Dx}\right)-U^{(k)}\left(\frac{x_i-X^k(t)}{\Dx}\right) - \frac{1}{\Dx} \int_{I_i} H^{(k)}(x) -H^{(k)}(x-X^k(t)) dx \right] \\
=&\ \Dx \sum_i  \left[ U^{(k)}\left(\frac{x_i-X^k(t)}{\Dx}\right) -\frac{1}{\Dx} \int_{I_i} H^{(k)}(x,t)) dx \right] \\
& -X^k(t)\bigl(u^{(k-1)}-u^{(k)}\bigr) - X^k(t)\bigl(u^{(k)}-u^{(k-1)}\bigr) \\
= &\ \Dx \sum_i  \left[ U^{(k)}\left(\frac{x_i-X^k(t)}{\Dx}\right) -\frac{1}{\Dx} \int_{I_i} H^{(k)}(x,t) dx \right] \\
= &\ \int_{\R}  V^{(k)}(x,t) - H^{(k)}(x,t) \ dx.
\end{align*}
Thus,
\begin{equation}\label{eq:mass_change_zero}
\int_{\R} H^{(k)}(x,t) - V^{(k)}(x,t) \ dx = \int_{\R} H^{(k)}(x) - V^{(k)}(x) \ dx = 0,
\end{equation}
so the difference $H^{(k)}(\cdot,t) - V^{(k)}(\cdot,t)$ satisfies \eqref{eq:wassercondition}, and hence the Wasserstein distance $W_1\bigl(H^{(k)}(\cdot,t) - V^{(k)}(\cdot,t)\bigr)$ is well-defined.

\begin{lemma}\label{lem:trav_wave_err}
For any $t \geq 0$,
\begin{align}\label{eq:trav_wav_err} 
W_1\left(H^{(k)}(\cdot,t),\ V^{(k)}(\cdot,t)\right) \leq C\Dx^2,
\end{align}
where $H^{(k)}$ and $V^{(k)}$ are defined in \eqref{eq:singlewaveentrsoln} and \eqref{eq:const_TW_initial}, respectively, and $C$ only depends on $u^{(k-1)}$ and $u^{(k)}$.
\end{lemma}
\begin{proof}
We write out the representation \eqref{eq:wassalternative} of the Wasserstein distance, apply a change-of-variables and split the integrals into negative and positive values of the integrated variable:
\begin{align*}
W_1&\Bigl(H^{(k)}(\cdot,t),\ V^{(k)}(\cdot,t)\Bigr) = \int_{\R} \left| \int_{-\infty}^x H^{(k)}\bigl(y-X^k(t)\bigr) - V^{(k)}(y,t) dy \right| dx \\
=&\ \int_{\R} \left| \int_{-\infty}^x H^{(k)}(y) - V^{(k)}\bigl(y+X^k(t),t\bigr) dy \right| dx \\
=&\ \int_{-\infty}^0 \left| \int_{-\infty}^x H^{(k)}(y) - V^{(k)}\bigl(y+X^k(t),t\bigr) dy \right| dx + \int_0^\infty \left| \int_{-\infty}^x H^{(k)}(y) - V^{(k)}\bigl(y+X^k(t),t\bigr) dy \right| dx \\
 = &\ \int_{-\infty}^0 \left| \int_{-\infty}^x u^{(k-1)} - V^{(k)}\bigl(y+X^k(t),t\bigr) dy \right| dx \\
&+ \int_0^\infty \left| \int_{-\infty}^0 u^{(k-1)} - V^{(k)}\bigl(y+X^k(t),t\bigr) dy + \int_{0}^x u^{(k)} - V^{(k)}\bigl(y+X^k(t),t\bigr) dy \right| dx.
\end{align*}
From \eqref{eq:mass_change_zero} we see that
\[
\int_{-\infty}^0 u^{(k-1)} - V^{(k)}\bigl(x+X^k(t),t\bigr) \ dx = \int_0^\infty V^{(k)}\bigl(x+X^k(t),t\bigr) - u^{(k)} \ dx,
\]
so the above can be written as
\begin{align*}
W_1&\Bigl(H^{(k)}(\cdot,t),\ V^{(k)}(\cdot,t)\Bigr) = \ \int_{-\infty}^0 \left| \int_{-\infty}^x u^{(k-1)} - V^{(k)}\bigl(y+X^k(t),t\bigr) dy \right| dx \\
& + \int_0^\infty \left| \int_{0}^{\infty} V^{(k)}\bigl(y+X^k(t),t\bigr) - u^{(k)} dy + \int_{0}^x u^{(k)} - V^{(k)}\bigl(y+X^k(t),t\bigr) dy \right| dx \\
= &\ \int_{-\infty}^0 \left| \int_{-\infty}^x u^{(k-1)} - V^{(k)}\bigl(y+X^k(t),t\bigr) dy \right| dx + \int_0^\infty \left| \int_{x}^{\infty} V^{(k)}\bigl(y+X^k(t),t\bigr) - u^{(k)} dy \right| dx \\
\leq &\ \int_{-\infty}^0 \int_{-\infty}^x \beta_k e^{\alpha_k \frac{y}{\Dx}} dy dx + \int_0^\infty \int_{x}^{\infty} \beta_k e^{-\alpha_k \frac{y}{\Dx}} dy dx \\
= &\ 2\frac{\beta_k}{\alpha_k^2}\Dx^2,
\end{align*}
where we have applied the exponential decay estimate \eqref{eq:exponentialdecay2}.
\end{proof}

We are now ready to prove the Main Theorem \ref{thm:main} for the special case $K=1$.
\begin{theorem}\label{thm:singleshockconv}
Let $u_0$ contain a single, negative jump as in \eqref{eq:heaviside}, let $u_\Dx$ be computed with the $W_1$-contractive, monotone finite volume method \eqref{eq:num_met}, assumed to be $W_1$-contractive, and let $u$ be the entropy solution. Then under the CFL condition \eqref{eq:cfl},
\begin{equation}
W_1\bigl(u_\Dx(\cdot, t^n), u(\cdot,t^n)\bigr) \leq C\Dx^2 \qquad \forall\ t^n>0
\end{equation}
for a constant $C>0$ only depending on $u^{(k-1)}$ and $u^{(k)}$.
\end{theorem}
\begin{proof}
The exact solution is $u(x,t) = H^{(k)}(x,t)$, as defined in \eqref{eq:singlewaveentrsoln}. Applying the triangle inequality and Lemma \ref{lem:trav_wave_err}, we get
\begin{align*}
W_1\bigl(u_\Dx(\cdot, t^n), u(\cdot,t^n)\bigr) &\leq W_1\Bigl(u_\Dx(\cdot, t^n), V^{(k)}(\cdot,t^n)\Bigr) + W_1\Bigl(V^{(k)}(\cdot,t^n), H^{(k)}(\cdot,t^n) \Bigr) \\
&\leq W_1\Bigl(u_\Dx(\cdot, t^n), V^{(k)}(\cdot,t^n)\Bigr) + C\Dx^2.
\end{align*}
By the stability estimate \eqref{eq:wass_err_nonhom} in Theorem \ref{thm:stability} and by Proposition \ref{prop:pos_err}, we have
\begin{align*}
W_1  \Bigl(u_\Dx(\cdot, t^n), V^{(k)}(\cdot,t^n)\Bigr) &\leq W_{1,D}\Bigl(u_\Dx(\cdot, 0), V^{(k)}\Bigr) \\
 & = W_{1}\Bigl(u_\Dx(\cdot, 0), V^{(k)}\Bigr) \\
 &\leq W_{1}\Bigl(u_\Dx(\cdot, 0), H^{(k)}(\cdot,0)\Bigr) + W_{1}\Bigl(H^{(k)}(\cdot,0), V^{(k)}\Bigr) 
\end{align*}
and by \eqref{eq:project_err} and Lemma \ref{lem:trav_wave_err} with $t=0$, the above is bounded by $C\Dx^2$.
\end{proof}

\section{Convergence rate for K initial jumps}\label{sec:multiple_jumps}
In this section we prove the Main Theorem \ref{thm:main} in its full generality, for any (finite) number of shocks $K\in\N$. We first consider only times $t\leq t_{(1)}$ (where $t_{(1)}$ is the time of the first shock interaction), and consider arbitrary times $t>0$ in Section \ref{sec:after_interaction}.

\subsection{Error before shock interactions}
The following lemma proves the Main Theorem \ref{thm:main} for any $K\in\N$, but before any shock interactions.
\begin{lemma}\label{lem:kshockerror}
Under the CFL condition \eqref{eq:cfl}, there is a $C_K>0$ such that
\begin{equation}\label{eq:kshockerror}
W_1\bigl(u(\cdot, t^n),u_\Dx(\cdot ,t^n) \bigr) \leq C_K\Dx^2
\end{equation}
for $0 \leq t^n < t_{(1)}$.
\end{lemma}
\begin{proof}
We use the intermediate solution $\bar{u}(x,t)$, similar to the one introduced in \cite{zhang_teng},
\begin{align*}
\bar{u}(x,t) &:=u(x,t)+\sum_{k=1}^K\left[V^{(k)}\left(x,t\right) - H^{(k)}(x,t)\right] \\
&=-\sum_{k=1}^{K-1}u^{(k)} + \sum_{k=1}^K V^{(k)}\left(x,t\right),
\end{align*}
defined for $0\leq t \leq t_{(1)}$. After time $t=t_{(1)}$, a new intermediate solution with $K-1$ terms (or fewer) is defined in the same manner. Here, as before, $X^{k}$ and $V^{(k)}$ are defined in \eqref{eq:characteristics} and \eqref{eq:const_TW}, respectively. Each modified discrete shock wave $V^{(k)}$ is chosen such that \eqref{eq:mass_change_zero} holds. 

By the triangle inequality, the error in \eqref{eq:kshockerror} can be bounded by
\[
W_1\bigl(u(\cdot, t^n),u_\Dx(\cdot ,t^n) \bigr) \leq W_1\bigl(u(\cdot,t^n),\bar{u}(\cdot,t^n)\bigr) + W_1\bigl(\bar{u}(\cdot,t^n),u_\Dx(\cdot ,t^n)\bigr).
\]
From repeatedly applying the triangle inequality and Lemma \ref{lem:trav_wave_err}, we easily obtain
\begin{align}\label{eq:err_accu_inter}
W_1\bigl(u(\cdot,t^n),\bar{u}(\cdot,t^n)\bigr)\leq C_K \Dx^2
\end{align}
for $0 \leq t^n < t_{(1)}$. Thus, to complete the proof, we need to show that 
\begin{align}\label{eq:err_numer_inter}
W_1\bigl(\bar{u}(\cdot,t^n),u_\Dx(\cdot ,t^n)\bigr)\leq C_K\Dx^2.
\end{align}
The intermediate solution $\bar{u}(x,t^n)$ satisfies the inhomogeneous finite volume method \eqref{eq:inhom_num} with the right hand side
\begin{align*}
h(x,t^n) =&\ \frac{1}{\Dx} \Bigl( F\bigl(\bar{u}(x, t^n),\, \bar{u}(x+\Dx, t^n)\bigr) -  F\bigl(\bar{u}(x-\Dx, t^n),\,\bar{u}(x, t^n)\bigr)\Bigr)\\
& - \frac{1}{\Dx}\sum_{k=1}^K\Bigl(F\bigl(V^{(k)}(x,t^n),\, V^{(k)}(x+\Dx,t^n)\bigr) - F\bigl(V^{(k)}(x-\Dx,t^n),\, V^{(k)}(x, t^n)\bigr)\Bigr)
\end{align*}
for $n=0,\dots,N-1$, where $N$ is the largest integer such that $t^N < t_{(1)}$. We can now use Theorem \ref{thm:stability} with the above $h$ and with $g=0$ to prove \eqref{eq:err_numer_inter}. The initial data is bounded by 
\begin{align*}
W_{1,D}\bigl(\bar{u}(\cdot,0),u_\Dx(\cdot ,0)\bigr) \leq \sum_{k=1}^K W_{1,D}\bigl(H_\Dx^{(k)}(\cdot,0),V^{(k)}(\cdot ,0)\bigr) \leq K C \Dx^2,
\end{align*}
where $H_\Dx^{(k)}(x,0)$ is the piecewise constant projection of $H^{(k)}(x-x^k)$. Each term in the sum above is bounded exactly in the same way as the initial data in the proof of Theorem \ref{thm:singleshockconv}.

So it remains to show that
\begin{equation}\label{eq:err_h_disc}
\Dt \sum_{n=0}^{N-1} W_{1,D}\bigl(h(\cdot,t^n), 0\bigr) \leq C_K \Dx^2.
\end{equation}
Thus, we need to estimate $W_{1,D}\bigl(h(\cdot,t^n), 0\bigr)$ in \eqref{eq:wass_err_nonhom} for each timestep $t^n\leq t^{N-1}$. 
We have
\begin{align*}
V^{(k)}\left(x\pm \Dx,t^n\right) &=\sum_{i} U^{(k)}\Bigl(\frac{x_i-X^{k}(t^n)}{\Dx}\Bigr)\ind_{I_i}(x\pm\Dx) \\
&= \sum_{i} U^{(k)}\Bigl(\frac{x_i-X^{k}(t^n)}{\Dx}\Bigr)\ind_{I_{i\mp 1}}(x) \\
&= \sum_{i} U^{(k)}\Bigl(\frac{x_i\pm\Dx-X^{k}(t^n)}{\Dx}\Bigr)\ind_{I_{i}}(x).
\end{align*}
Denote $V^{(k),n}_i=V^{(k)}(x_i,t^n)$ and $\bar{u}_i^n=\bar{u}(x_i,t^n)=-\sum_{k=1}^{K-1}u^{(k)}+\sum_{k=1}^K V^{(k),n}_i$. By adding and subtracting $\sum_{k=1}^{K-1}f(u^{(k)})$ and applying summation by parts, we get
\begin{align*}
& W_{1,D}\bigl(h(\cdot,t^n), 0\bigr) \\
&= \sup_{\|\phi\|_\DLip\leq 1} \sum_i \phi_i\left[ F\left(\bar{u}_i^n,\bar{u}_{i+1}^n\right) -  F\left(\bar{u}_{i-1}^n,\bar{u}_i^n\right) - \sum_{k=1}^K\left(F\bigl(V^{(k),n}_i,V^{(k),n}_{i+1}\bigr) - F\bigl(V^{(k),n}_{i-1},V^{(k),n}_{i}\bigr) \right)\right] \\
&= \sup_{\|\phi\|_\DLip\leq 1} \sum_i \phi_i\Biggl[ \biggl(F\bigl(\bar{u}_i^n,\bar{u}_{i+1}^n\bigr) - \sum_{k=1}^K F\bigl(V^{(k),n}_i,V^{(k),n}_{i+1}\bigr) + \sum_{k=1}^{K-1}f\bigl(u^{(k)}\bigr)\biggr) \\*
&\qquad\qquad\qquad\qquad  -  \biggl(F\bigl(\bar{u}_{i-1}^n,\bar{u}_i^n\bigr) - \sum_{k=1}^K F\bigl(V^{(k),n}_{i-1},V^{(k),n}_{i}\bigr) + \sum_{k=1}^{K-1}f\bigl(u^{(k)}\bigr)\biggr) \Biggr] \\
&= \sup_{\|\phi\|_\DLip\leq 1} -\sum_i(\phi_{i+1}-\phi_i)\left[F\left(\bar{u}_i^n,\bar{u}_{i+1}^n\right) - \sum_{k=1}^K F\bigl(V^{(k),n}_i,V^{(k),n}_{i+1}\bigr) + \sum_{k=1}^{K-1}f\bigl(u^{(k)}\bigr)\right] \\
&= \sup_{\|\psi\|_{\ell^\infty}\leq 1} \Dx \sum_i \psi_i \left[F\left(\bar{u}_i^n,\bar{u}_{i+1}^n\right) - \sum_{k=1}^K F\bigl(V^{(k),n}_i,V^{(k),n}_{i+1}\bigr) + \sum_{k=1}^{K-1}f\bigl(u^{(k)}\bigr)\right] \\
&= \Dx \sum_i \left|F\left(\bar{u}_i^n,\bar{u}_{i+1}^n\right) - \sum_{k=1}^K F\bigl(V^{(k),n}_i,V^{(k),n}_{i+1}\bigr) + \sum_{k=1}^{K-1}f\bigl(u^{(k)}\bigr)\right|.
\end{align*}
%Here, we have replaced $\phi_{i+1}-\phi_i$ by $-\Dx\psi_i$ for an arbitrary $\psi\in\ell^\infty$ with $\|\psi\|_{\ell^\infty}\leq1$. 
Let
\begin{align*}
Z^s (t) = \frac{1}{2}\left( X^s(t)+X^{s+1}(t) \right), \quad s=1,\dots,K-1.
\end{align*}
Continuing from above and denoting $\bar{u}^n(x)=\bar{u}(x,t^n)$ and $V^{(k),n}(x)=V^{(k)}(x,t^n)$, we get
\begin{align*}
W_{1,D}\bigl(h(\cdot, t^n), 0\bigr) &= \int_\R \left|F\bigl(\bar{u}^n(x),\bar{u}^n(x+\Dx)\bigr) - \sum_{k=1}^K F\bigl(V^{(k),n}(x),V^{(k),n}(x+\Dx)\bigr) + \sum_{k=1}^{K-1}f\bigl(u^{(k)}\bigr)\right| dx \\
&= \sum_{s=0}^{K-1} \mathcal{J}^{(s),n},
\end{align*}
where we have split the integration domain over 
\[
\mathcal{J}^{(0),n}=\int_{-\infty}^{Z^1(t^n)}\dots dx, \qquad \mathcal{J}^{(s),n}=\int_{Z^s(t^n)}^{Z^{s+1}(t^n)}\dots dx,\qquad \mathcal{J}^{(K-1),n}=\int_{Z^{K-1}(t^n)}^\infty\dots dx
\]
(for $s=1,\dots,K-2$). Let $F_a^{(k,l)}(x)$, $F_b^{(k,l)}(x)$, $F_a^{(s)}(x)$ and $F_b^{(s)}(x)$ (where $l$ is either $k$ or $k-1$) be such that
\[
f\bigl(u^{(l)}\bigr) - F\bigl(V^{(k),n}(x), V^{(k),n}(x+\Dx)\bigr) = F_a^{(k,l)}(x)\Bigl(u^{(l)}-V^{(k),n}(x)\Bigr) + F_b^{(k,l)}(x) \Bigl(u^{(l)} - V^{(k),n}(x+\Dx)\Bigr) \\
\]
and
\begin{align*}
& F\bigl(\bar{u}^n(x),\bar{u}^n(x+\Dx)\bigr) - F\bigl(V^{(s),n}(x), V^{(s),n}(x+\Dx)\bigr) \\
&= F_a^{(s)}(x)\Bigl(\bar{u}^n(x)-V^{(s),n}(x)\Bigr) + F_b^{(s)}(x) \Bigl(\bar{u}^n(x+\Dx)-V^{(s),n}(x+\Dx)\Bigr).
\end{align*}
Specifically, we can write
\begin{align*}
F_a^{(k,l)}(x) &= \int_0^1 \frac{\partial F}{\partial a}\left((1-\alpha )u^{(l)} + \alpha V^{(k),n}(x),\ (1-\alpha)u^{(l)} + \alpha V^{(k),n}(x+\Dx) \right)  d\alpha \\
F_b^{(k,l)}(x) &=  \int_0^1 \frac{\partial F}{\partial b}\left((1-\alpha )u^{(l)} + \alpha V^{(k),n}(x),\ (1-\alpha )u^{(l)} + \alpha V^{(k),n}(x+\Dx) \right)  d\alpha \\
F_a^{(s)}(x)   &=  \int_0^1 \frac{\partial F}{\partial a} \left((1-\alpha )\bar{u}^n(x) + \alpha V^{(s),n}(x),\ (1-\alpha )\bar{u}^n(x + \Dx) + \alpha V^{(s),n}(x+\Dx) \right)  d\alpha \\
F_b^{(s)}(x)   &= \int_0^1 \frac{\partial F}{\partial b}\left((1-\alpha )\bar{u}^n(x) + \alpha V^{(s),n}(x),\ (1-\alpha )\bar{u}^n(x + \Dx) + \alpha V^{(s),n}(x+\Dx) \right) d\alpha.
\end{align*}
Let us first consider the first interval, $\mathcal{J}^{(0),n}$. The last interval $\mathcal{J}^{(K-1),n}$ can be treated similarly. 
\begin{align*}
\mathcal{J}^{(0),n} &= \int_{-\infty}^{Z^{1}(t^n)} \left|F\bigl(\bar{u}^n(x),\bar{u}^n(x+\Dx)\bigr) - \sum_{k=1}^K F\bigl(V^{(k),n}(x),V^{(k),n}(x+\Dx)\bigr) + \sum_{k=1}^{K-1}f\bigl(u^{(k)}\bigr) \right| dx \\
= &\ \int_{-\infty}^{Z^{1}(t^n)} \Biggl| \sum_{k=2}^K \biggl[F_a^{(k,k-1)}(x) \Bigl(V^{(k),n}(x) - u^{(k-1)} \Bigr) + F_b^{(k,k-1)}(x)\Bigl(V{(k)}(x+\Dx,t^n) - u^{(k-1)} \Bigr) \biggr]   \\
&+ F_a^{(1)}(x) \Bigl(V^{(1),n}(x) - \bar{u}^n(x)\Bigr) + F_b^{(1)}(x) \Bigl(V^{(1),n}(x+\Dx) - \bar{u}^n(x+\Dx)\Bigr) \Biggr| dx \\
= &\ \int_{-\infty}^{Z^{1}(t^n)} \Biggl| \sum_{k=2}^K \biggl[F_a^{(k,k-1)}(x) \left(V^{(k),n}(x) - u^{(k-1)} \right) + F_b^{(k,k-1)}(x) \left(V^{(k),n}(x+\Dx) - u^{(k-1)} \right)\biggr] \\
&- F_a^{(1)}(x) \sum_{k=2}^K \left(V^{(k),n}(x) - u^{(k-1)} \right) - F_b^{(1)}(x) \sum_{k=2}^K\left(V^{(k),n}(x+\Dx) - u^{(k-1)} \right) \Biggr| dx \\
= &\ \int_{-\infty}^{Z^{1}(t^n)} \Biggl| \sum_{k=2}^K\biggl[ \left(F_a^{(k,k-1)}(x) - F_a^{(1)}(x)\right) \left(V^{(k),n}(x) - u^{(k-1)} \right) \\
& \qquad\qquad\quad + \left(F_b^{(k,k-1)}(x) - F_b^{(1)}(x) \right)\left(V^{(k),n}(x+\Dx) - u^{(k-1)} \right) \biggr]\Biggr| dx
\end{align*}
All values of $\bar{u}$, $V^{(k)}$ and $u^{(k)}$ lie in the bounded interval $\bigl[u^{(K)},\, u^{(0)}\bigr]$. Thus, $F_a^{(k,l)}(x)$, $F_b^{(k,l)}(x)$, $F_a^{(s)}(x)$ and $F_b^{(s)}(x)$ are bounded for all $k=1, \dots, K$. Therefore,
\[
\mathcal{J}^{(0),n} \leq \sum_{k=2}^K C_k \int_{-\infty}^{Z^{1}(t^n)}\left| V^{(k),n}(x) - u^{(k-1)}\right| + \left|V^{(k),n}(x+\Dx) - u^{(k-1)} \right| dx.
\]
Since $Z^{1}(t^n)\leq X^k(t^n)$ for $k=2,\dots,K$ for all $t^n\leq t_{(1)}$, we can apply \eqref{eq:exponentialdecay2A} to both terms in the above integrand and obtain
\begin{align*}
&\mathcal{J}^{(0),n} \leq \sum_{k=2}^K C_k \beta_k \int_{-\infty}^{Z^{1}(t^n)} \exp\left(-\alpha_k\left|\frac{x-X^{k}(t^n)}{\Dx}\right|\right) + \exp \left(-\alpha_k\left|\frac{x+\Dx-X^{k}(t^n)}{\Dx}\right|\right) dx \\
&= \sum_{k=2}^K C_k \beta_k \Biggl[\int_{-\infty}^{Z^{1}(t^n)} \exp\left(-\alpha_k\left|\frac{x-X^{k}(t^n)}{\Dx}\right|\right) dx \\
&\qquad\qquad + \int_{-\infty}^{Z^{1}(t^n)} \exp \left(-\alpha_k\left|\frac{x-X^{k}(t^n)}{\Dx}\right|\right) dx + \int_{Z^{1}(t^n)}^{Z^{1}(t^n)+\Dx} \exp \left(-\alpha_k\left|\frac{x-X^{k}(t^n)}{\Dx}\right|\right) dx \\
&= \Dx\sum_{k=2}^K C_k\beta_k\Biggl[\frac{2}{\alpha_k}\exp \left(-\alpha_k\left|\frac{Z^1(t^n)-X^{k}(t^n)}{\Dx}\right|\right) + \exp \left(-\alpha_k\left|\frac{Z^1(t^n)-X^{k}(t^n)}{\Dx}+\theta^{(0),n}_k\right|\right)\Biggr]
%&\leq \Dx\sum_{k=2}^K \frac{C_k \beta_k}{\alpha_k}\left[\exp\left(-\alpha_k\left|\frac{Z^1(t^n)-X^{k}(t^n)}{\Dx}\right|\right) + \exp \left(-\alpha_k\left|\frac{Z^1(t^n)-X^{k}(t^n)}{\Dx}\right|\right)\right]
\end{align*}
for some $\theta^{(0),n}_k\in[0,1]$, by the mean value theorem applied to the last integral. By a similar argument we get
\[
\mathcal{J}^{(K-1),n} \leq \Dx\sum_{k=1}^{K-1} C_k\beta_k \frac{2}{\alpha_k}\exp \left(-\alpha_k\left|\frac{Z^{K-1}(t^n)-X^{k}(t^n)}{\Dx}\right|\right)
\]

Next, let us look at the intermediate intervals. The approach is similar to the above. We start by splitting the sum in a specific way,
\begin{align*}
\mathcal{J}^{(s),n} =&\ \int_{Z^{s}(t^n)}^{Z^{s+1}(t^n)}  \left|F\bigl(\bar{u}^n(x),\bar{u}^n(x+\Dx)\bigr) - \sum_{k=1}^K F\bigl(V^{(k),n}(x),V^{(k),n}(x+\Dx)\bigr) + \sum_{k=1}^{K-1}f\bigl(u^{(k)}\bigr) \right| dx \\
=&\ \int_{Z^{s}(t^n)}^{Z^{s+1}(t^n)} \left| \sum_{k>s+1}^K \biggl[ F_a^{(k,k-1)}(x)\left(V^{(k),n}(x) - u^{(k-1)} \right) + F_b^{(k,k-1)}(x) \left(V^{(k),n}(x+\Dx) - u^{(k-1)} \right) \biggr] \right. \\
& + \sum_{k<s+1} \biggl[ F_a^{(k,k)}(x) \left(V^{(k),n}(x) - u^{(k)} \right) + F_b^{(k,k)}(x) \left(V^{(k),n}(x+\Dx) - u^{(k)} \right)  \biggr] \\
& + F_a^{(s+1)}(x)\left(V^{(s+1),n}(x) - \bar{u}^n(x) \right) + F_b^{(s+1)}(x) \left(V^{(s+1),n}(x+\Dx) - \bar{u}^n(x+\Dx) \right) \Biggr| dx \\
=&\ \int_{Z^{s}(t^n)}^{Z^{s+1}(t^n)} \Biggl| \sum_{k>s+1}^K \biggl[ \left( F_a^{(k,k-1)}(x)-F_a^{(s+1)}(x) \right)\left(V^{(k),n}(x) - u^{(k-1)} \right)  \\
& \qquad\qquad\qquad+ \left(F_b^{(k,k-1)}(x)-F_b^{(s+1)}(x) \right)\left(V^{(k),n}(x+\Dx) - u^{(k-1)} \right) \biggr] \\
& + \sum_{k<s+1} \biggl[\left(F_a^{(k,k)}(x)-F_a^{(s+1)}(x) \right)\left(V^{(k),n}(x) - u^{(k)} \right)  \\
& \qquad+ \left(F_b^{(k,k)}(x)-F_b^{(s+1)}(x)\right)\left(V^{(k),n}(x+\Dx) - u^{(k)} \right) \biggr] \Biggr| dx \\
\leq&\ \int_{-\infty}^{Z^{s+1}(t^n)} \Biggl| \sum_{k>s+1}^K \left( F_a^{(k,k-1)}(x)-F_a^{(s+1)}(x) \right)\left(V^{(k),n}(x) - u^{(k-1)} \right) \\
&  \qquad\qquad\qquad+\left(F_b^{(k,k-1)}(x)-F_b^{(s+1)}(x)\right)\left(V^{(k),n}(x+\Dx) - u^{(k-1)} \right) \Biggr| dx  \\
& + \int_{Z^{s}(t^n)}^{\infty} \Biggl|\sum_{k<s+1} \left(F_a^{(k,k)}(x)-F_a^{(s+1)}(x)\right)\left(V^{(k),n}(x) - u^{(k)} \right) \\
& \qquad\qquad\qquad+ \left(F_b^{(k,k)}(x)-F_b^{(s+1)}(x) \right)\left(V^{(k),n}(x+\Dx) - u^{(k)} \right)  \Biggr| dx \\
=&\ \mathcal{J}^{(s),n}_1 + \mathcal{J}^{(s),n}_2.
\end{align*}
For $\mathcal{J}^{(s),n}_2$ we have
\begin{align*}
\mathcal{J}^{(s),n}_2 & \leq \sum_{k<s+1}C_k\int_{Z^{s}(t^n)}^{\infty} \left| V^{(k),n}(x) - u^{(k)} \right| + \left| V^{(k),n}(x+\Dx) - u^{(k)} \right| dx
\end{align*}
Since $Z^{s}(t^n) \geq X^{k}(t^n)$ for $k=1,\dots,s$, we can use \eqref{eq:exponentialdecay2B} to conclude that
\begin{align*}
\mathcal{J}^{(s),n}_2 \leq 2\Dx\sum_{k<s+1} C_k \frac{\beta_k}{\alpha_k} \exp \left(-\alpha_k\frac{Z^{s}(t^n)-X^{k}(t^n)}{\Dx}\right).
\end{align*}
Similarly to the estimate of $\mathcal{J}^{(0),n}$, we use the fact that $X^{k}(t^n) \geq Z^{s+1}(t^n)$ for $k>s+1$ to get
\begin{align*}
\mathcal{J}^{(s),n}_1 \leq \Dx\sum_{k>s+1}^K C_k\beta_k\biggl[\frac{2}{\alpha_k} \exp \left(-\left|\alpha_k\frac{Z^{s+1}(t^n)-X^{k}(t^n)}{\Dx}\right|\right) + \exp\left(-\alpha_k\left|\frac{Z^{s+1}(t^n) -X^{k}(t^n)}{\Dx}+\theta_k^{(s),k} \right|\right)\biggr],
\end{align*}
for some $\theta_k^{(s),k}\in[0,1]$.

We now need to sum up the error of these integrals in each timestep. We have
\begin{align*}
\Dt \sum_{n=0}^{N-1} W_{1,D}\bigl(h(x,t^n), 0\bigr) &\leq \Dt \sum_{n=0}^{N-1} \sum_{s=0}^{K-1} \mathcal{J}^{(s),n} \\
& \leq \Dt \sum_{n=0}^{N-1} \left( \mathcal{J}^{(0),n} + \mathcal{J}^{(K-1),n} + \sum_{s=1}^{K-2} \mathcal{J}^{(s),n}_1 + \mathcal{J}^{(s),n}_2 \right).
\end{align*}
Let 
\begin{align*}
t_{s,k}=\frac{x^{k}-x^{s}}{D^s-D^k},
\end{align*}
which is the interaction time of two shocks, and recall that $t^{N-1} < t_{(1)} \leq t_{s,k}$ and $Z^{s+1}(t^n) < X^{k}(t^n)$ for $k=s+2,\dots, K$. We estimate $\mathcal{J}^{(s)}_1$ as follows,
\begin{align*}
&\Dt \sum_{n=0}^{N-1} \mathcal{J}^{(s),n}_1 \\
&\leq \Dt\Dx \sum_{n=0}^{N-1}\sum_{k>s+1}^KC_k\beta_k\Biggl[ \frac{2}{\alpha_k} \exp \left(-\alpha_k\left|\frac{Z^{s+1}(t^n)-X^{k}(t^n)}{\Dx}\right|\right) + \exp\left(-\alpha_k\left|\frac{Z^{s+1}(t^n) -X^{k}(t^n)}{\Dx}+\theta^{(s),n}_k \right|\right)\Biggr] \\
&\leq \Dx\sum_{k>s+1}^K C_k\beta_k\int_0^{t_{s+1,k}}\Biggl[\frac{2}{\alpha_k} \exp \left(-\alpha_k\left|\frac{Z^{s+1}(t)-X^{k}(t)}{\Dx}\right|\right) + \exp\left(-\alpha_k\left|\frac{Z^{s+1}(t) -X^{k}(t)}{\Dx}+\theta^{(s),n}_k \right|\right)\Biggr] dt \\
&\leq \Dx\sum_{k>s+1}^K C_k\beta_k\Biggl[\frac{2\Dx}{\alpha_k^2(D^{s+1}+D^{s+2}-2D^k)} +  \frac{\Dx}{\alpha_k(D^{s+1}+D^{s+2}-2D^k)}\Biggr]\\
&= O(\Dx^2). 
\end{align*}
By similar treatment of the other integrals, we get \eqref{eq:err_h_disc}.
\end{proof}

\subsection{Error after shock interaction}\label{sec:after_interaction}
We can now conclude the proof of the Main Theorem.
\begin{proof}[Proof of Main Theorem \ref{thm:main}]
To conclude that \eqref{eq:main_conv_res} holds for all timesteps $t^n$ and for any number of shocks, we use induction on $K$. We showed in Theorem \ref{thm:singleshockconv} that Theorem \ref{thm:main} holds for $K=1$. We assume that Theorem~\ref{thm:main} holds for initial data with at most $K-1 \geq 1$ shocks. Let $u_0$ have $K$ initial shocks and fix $N\in\N$ such that $t^N < t_{(1)} \leq t^{N+1}$. By Lemma \ref{lem:kshockerror}, the result is true up to time $t^n=t^N$, so it suffices to consider times $t^n\geq t^{N+1}$.

Let $v_\Dx^{N+1}(x, t)$ be the numerical solution with initial data $v_0^{N+1}(x)=u(x,t^{N+1})$ and let $v_\Dx^N(x, t)$ be the numerical solution with initial data $v_0^N(x)=u(x,t^N)$. By the induction hypothesis we have
\[
W_1\Bigl(u\bigl(\cdot, t^{N+1}+t^m\bigr), v_\Dx^{N+1}\bigl(\cdot, t^m\bigr)\Bigr) \leq C_{K-1} \Dx^2
\]
for any $m\geq0$, and by the stability Theorem \ref{thm:stability}, we have
\[
W_1\Bigl(v_\Dx^{N+1}(\cdot, t^m), u_\Dx(\cdot, t^{N+1} + t^m) \Bigr) \leq W_{1,D}\Bigl(v_\Dx^{N+1}(\cdot, 0), u_\Dx(\cdot, t^{N+1}) \Bigr).
\] 
Thus, the error at time $t^{N+1}+t^m$ is
\begin{align*}
W_1\Bigl(&u(\cdot, t^{N+1}+t^m), u_\Dx(\cdot, t^{N+1} + t^m) \Bigr) \\*
&\leq W_1\Bigl(u(\cdot, t^{N+1}+t^m), v_\Dx^{N+1}(\cdot, t^m) \Bigr) + W_1\Bigl(v_\Dx^{N+1}(\cdot, t^m), u_\Dx(\cdot, t^{N+1} + t^m) \Bigr) \\
&\leq C_{K-1}\Dx^2 + W_{1,D}\Bigl(v_\Dx^{N+1}(\cdot, 0), u_\Dx(\cdot, t^{N+1}) \Bigr) \\
&\leq C_{K-1}\Dx^2 + \underbrace{W_{1,D}\Bigl(v_\Dx^{N+1}(\cdot, 0), v_\Dx^N(\cdot, \Dt) \Bigr)}_{=\mathcal{E}_1} + \underbrace{W_{1,D}\Bigl(u^N_{\Dx}(\cdot, \Dt), u_\Dx(\cdot, t^{N+1}) \Bigr)}_{=\mathcal{E}_2}.
\end{align*}
Let $v_i^N$ and $v_i^{N+1}$ be grid values so that $v_\Dx^N(x,0) = \sum_i v_i^N\ind_{I_i}(x)$ and $v_\Dx^{N+1}(x,0) = \sum_i v_i^{N+1}\ind_{I_i}(x)$. Since
\[
v_i^{N+1} = \frac{1}{\Dx} \int_{I_i} u(x,t^{N+1})\ dx = v_i^N - \frac{1}{\Dx}\int_{0}^{\Dt} f(u(x_{i+\hf},t^N+t)) - f(u(x_{i-\hf},t^N+t)) \ dt,
\]
we get
\begin{align*}
\mathcal{E}_1 =&\ \sup_{\| \phi \|_\DLip\leq 1}\Dx \sum_i \phi_i \left( v_i^{N+1}-v_i^N + \frac{\Dt}{\Dx}\left(F(v_i^N,v_{i+1}^N) - F(v_{i-1}^N,v_i^N) \right) \right) \\
%= &\ \sup_{\| \phi \|_\DLip\leq 1}\Dx \sum_i \phi_i \left( \frac{1}{\Dx} \int_{I_i} u(x,t^{N+1}) - u(x,t^N) dx + \frac{\Dt}{\Dx}\left(F(v_i^N,v_{i+1}^N) - F(v_{i-1}^N,v_i^N) \right)\right) \\
= &\ \sup_{\| \phi_{\Dx} \|_\DLip\leq 1}\Dx \sum_i \phi_i \left( -\frac{1}{\Dx} \int_{0}^{\Dt} f(u(x_{i+\hf},t^N+t)) - f(u(x_{i-\hf},t^N+t)) \ dt \right.\\
& \left. + \frac{\Dt}{\Dx}\left(F(v_i^N,v_{i+1}^N) - F(v_{i-1}^N,v_i^N) \right)\right) \\
%= & \sup_{\| \phi \|_\DLip\leq 1} 
%-\lim_{i \to \infty} \phi_i \int_{0}^{\Dt} f(u(x_{i+\hf},t^N+t))\ dt + \lim_{i \to -\infty} \phi_i \int_{0}^{\Dt} f(u(x_{i-\hf},t^N+t))\ dt \\
%& + \lim_{i \to \infty} \Dt\phi_i F(v_i^N,v_{i+1}^N) - \lim_{i \to -\infty} \Dt \phi_i F(v_{i-1}^N,v_i^N) \\
%\sum_i (\phi_{i+1}-\phi_i)\left(\int_0^{\Dt} f(u(x_{i+\hf},t^N+t))\ dt - \Dt F(v_i^N,v_{i+1}^N)\right) \\
= &\ \sup_{\| \phi_{\Dx} \|_\DLip\leq 1} \sum_i (\phi_{i+1}-\phi_i)\left(\int_0^{\Dt} f(u(x_{i+\hf},t^N+t)) - F(v_i^N,v_{i+1}^N) \ dt \right) \\
%\end{align*}
%as
%\begin{align*}
%\lim_{i \to \infty} & \Dt\phi_i F(v_i^N,v_{i+1}^N) - \int_{0}^{\Dt} f(u(x_{i+\hf},t^N+t))\ dt \\
%& = \lim_{i \to \infty} \phi_i \int_{0}^{\Dt} F(v_i^N,v_{i+1}^N) - f(u(x_{i+\hf},t^N+t))\ dt \\
%& = \lim_{i \to \infty} \phi_i \int_{0}^{\Dt} F(u^K,u^K) - f(u^K)\ dt = \lim_{i \to \infty} \phi_i {\Dt} \left(f(u^K) - f(u^K)\right) = 0,
%\end{align*}
%and similarly for $i \to -\infty$. Hence,
%\begin{align*}
%W_{1,D} & \Bigl(v_\Dx^{N+1}(\cdot, 0), v_\Dx^N(\cdot, \Dt) \Bigr) \\
%\leq&\ \sup_{\| \phi_{\Dx} \|_\DLip\leq 1} \sum_i (\phi_{i+1}-\phi_i)\left(\int_0^{\Dt} f(u(x_{i+\hf},t^N+t))\ dt - F(v_i^N,v_{i+1}^N)\right) \\
=&\ \Dx \sum_i \left| \int_0^{\Dt} f(u(x_{i+\hf},t^N+t)) - F(v_i^N,v_{i+1}^N) \ dt \right|.
\end{align*}
Denote by $\mathcal{S}\subset\Z$ the set of indices $i$ where $v_i^N\neq v_{i+1}^N$. This set has at most $2K$ elements, so
\begin{align*}
\mathcal{E}_1 =&\ \Dx \sum_{i\in\mathcal{S}} \left| \int_0^{\Dt} f(u(x_{i+\hf},t^N+t)) - F(v_{i}^N,v_{i+1}^N) \ dt \right|\\
&+ \Dx \sum_{i\notin\mathcal{S}} \left| \int_0^{\Dt} f(u(x_{i+\hf},t^N+t)) - F(v_{i}^N,v_{i}^N) \ dt \right| \\
\leq&\ \Dx \Dt C{2K} + \Dx \sum_{i\notin\mathcal{S}} \left| \int_0^{\Dt} f(u(x_{i+\hf},t^N+t)) - f(v_{i}^N) \ dt \right| \\
= &\ \Dx \Dt C{2K},
\end{align*}
the last equality following from the fact that  $u(x_{i+\hf},t^N+t) = v_{i}^N$ for $0\leq t < \Dt$ for all $i \notin \mathcal{S}$. Moreover, by Theorem \ref{thm:stability} and the proof of Lemma \ref{lem:kshockerror},
\begin{align*}
\mathcal{E}_2 &\leq W_{1,D}\bigl(v_\Dx^N(\cdot, 0), u_\Dx(\cdot, t^N)\bigr) \\
&\leq W_{1,D}\bigl(v_\Dx^N(\cdot, 0), \bar{u}(\cdot, t^N)\bigr) + W_{1,D}\bigl(\bar{u}(\cdot, t^N), u_\Dx(\cdot, t^N)\bigr) \leq C\Dx^2.
\end{align*}
%\begin{align*}
%\mathcal{E}_2 =&\ W_{1,D}\left(v_\Dx^N(\cdot, \Dt), u_\Dx(\cdot, t^{N+1}) \right) 
%\leq W_{1,D}\left(v_\Dx^N(\cdot, 0), u_\Dx(\cdot, t^N) \right) \\
%\leq &\ W_{1,D}\left(v_\Dx^N(\cdot, 0), \bar{u}(\cdot, t^N) \right) + W_{1,D}\left(\bar{u}(\cdot, t^N), u_\Dx(\cdot, t^N) \right) \\
%\leq &\ W_{1}\left(v_\Dx^N(\cdot, 0), u(x,t^N)) \right) + \sum_{i=1}^K W_{1,D}\left(H^{(k)}(\cdot-X^{k}(t^N)), V^{(k)}(\cdot, t^N)\right) \\
%& + W_{1,D}\left(\bar{u}(\cdot, 0), u_\Dx(\cdot, 0) \right) + \Dt \sum_{n=0}^{N-1} W_{1,D}(h(x,t^n), 0) \\
%= &\ W_{1}\left(v_\Dx^N(\cdot, 0), u(x,t^N)) \right) + \sum_{i=1}^K W_{1}\left(H^{(k)}(\cdot-X^{k}(t^N)), V^{(k)}(\cdot, t^N)\right) \\
%& + W_{1}\left(\bar{u}(\cdot, 0), u_\Dx(\cdot, 0) \right) + \Dt \sum_{n=0}^{N-1} W_{1,D}(h(x,t^n), 0) \leq C \Dx^2.
%\end{align*}
We can conclude that 
\begin{align*}
W_1\bigl(u(\cdot, t^n), u_{\Dx}(\cdot, t^n) \bigr) \leq C_K \Dx^2
\end{align*}
holds for \emph{all} $t^n$, thus finishing the proof of Theorem \ref{thm:main}.
\end{proof}

\section{Numerical experiments}
\label{sec:numerical}
To illustrate the main theorem, we look at a numerical approximation of Burgers' equation on the interval $[0,1]$, 
\begin{align}\label{eq:burgers}
u_t + \left(\frac{u^2}{2}\right)_x = 0,
\end{align}
with initial data containing two jumps,
\begin{align}\label{eq:numinitialcond}
u_0(x) =\begin{cases}
    2       & \quad x<0.25\\
    1       & \quad 0.25 \leq x < 0.5 \\
    0       & \quad x \geq 0.5.
  \end{cases}
\end{align}
We use the Godunov scheme, i.e.\ the monotone scheme \eqref{eq:num_met} with Godunov flux function \eqref{eq:godflux}, and a CFL number of $0.3$. The exact solution is
\begin{align*}
u(x,t) =\begin{cases}
    2       & \quad x<0.25+1.5t\\
    1       & \quad 0.25+1.5t \leq x < 0.5+0.5t \\
    0       & \quad x \geq 0.5+0.5t \\
  \end{cases}
\end{align*}
for $t<0.25$ and 
\begin{align*}
u(x,t) =\begin{cases}
    2       & \quad x<3/8+t\\
    0       & \quad x \geq 3/8+t \\
  \end{cases}
\end{align*}
for $t\geq0.25$. At $t=0.25$ the two shocks interact. The initial condition is plotted in Figure \ref{fig:example}(a). In Figure \ref{fig:example}(b)--(c) we see the exact solution and the numerical approximation before ($t=0.15$) and after ($t=0.3$) shock interaction. Tables \ref{tab:god_before} and \ref{tab:god_after} show the observed rate of convergence of the numerical approximation before and after the shock interaction; it is clear that the $W_1$ error is $O(\Dx^2)$, as claimed. The $L^1$ error is $O(\Dx)$, as was shown in \cite{zhang_teng}.

\begin{figure}
\centering
\subfigure[Initial condition.]{\includegraphics[width=0.3\textwidth]{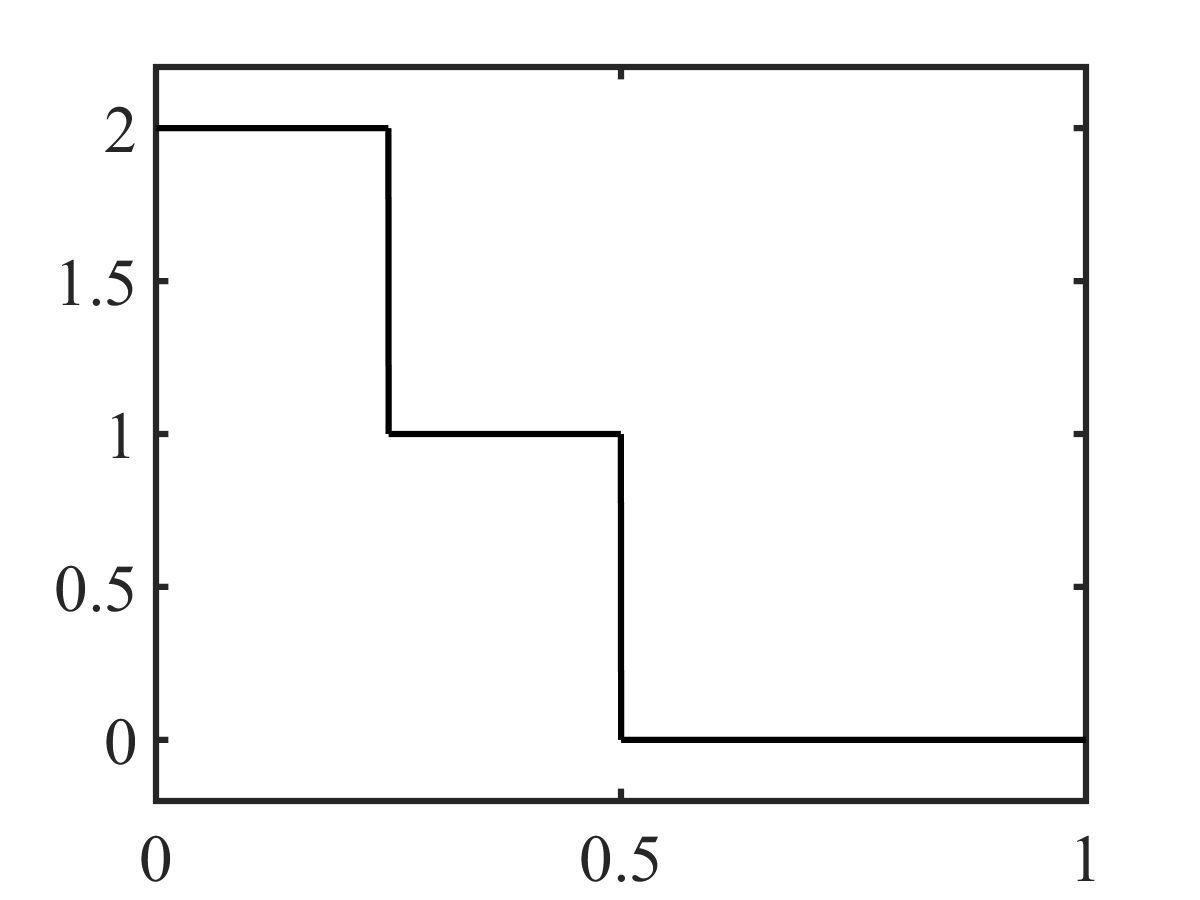}}
\subfigure[Solution at $t=0.15$.]{\includegraphics[width=0.3\textwidth]{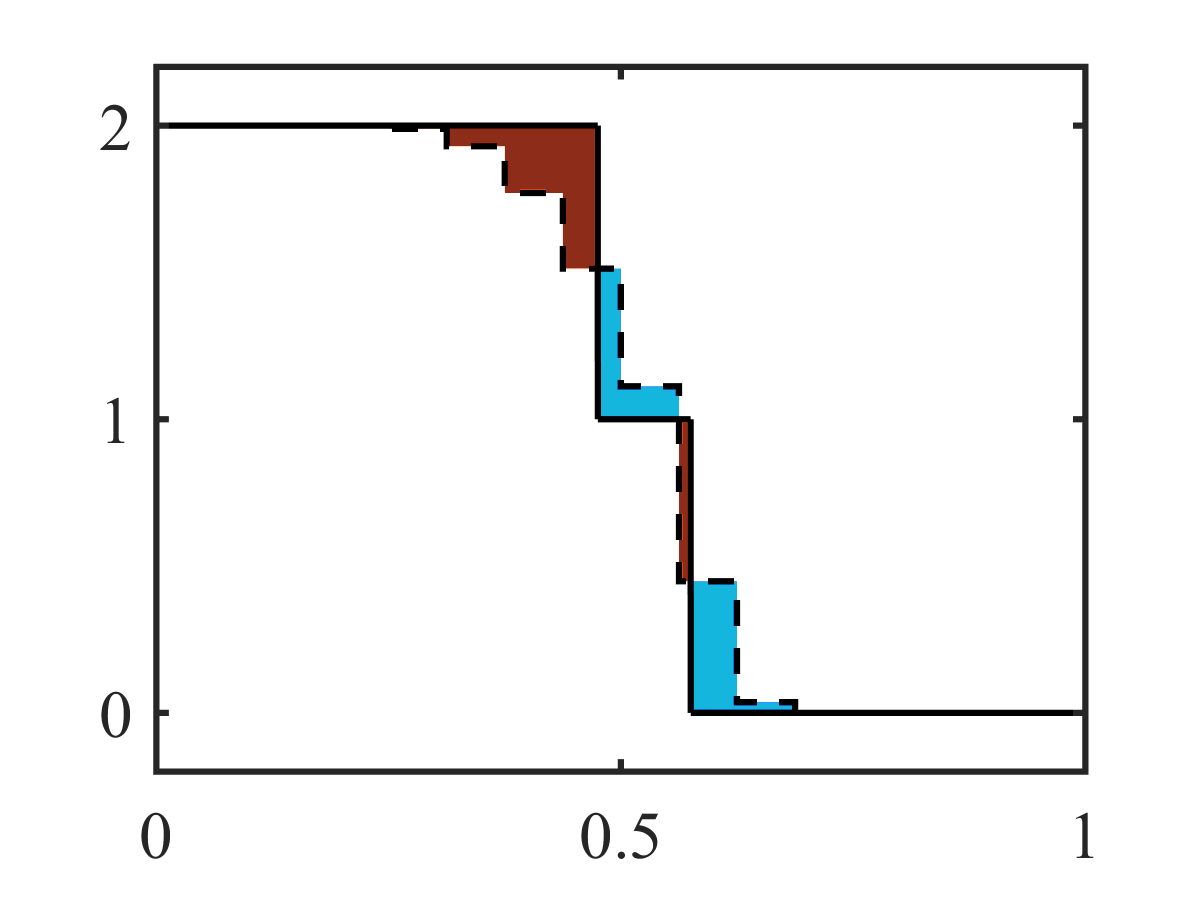}}
\subfigure[Solution at $t=0.3$.]{\includegraphics[width=0.3\textwidth]{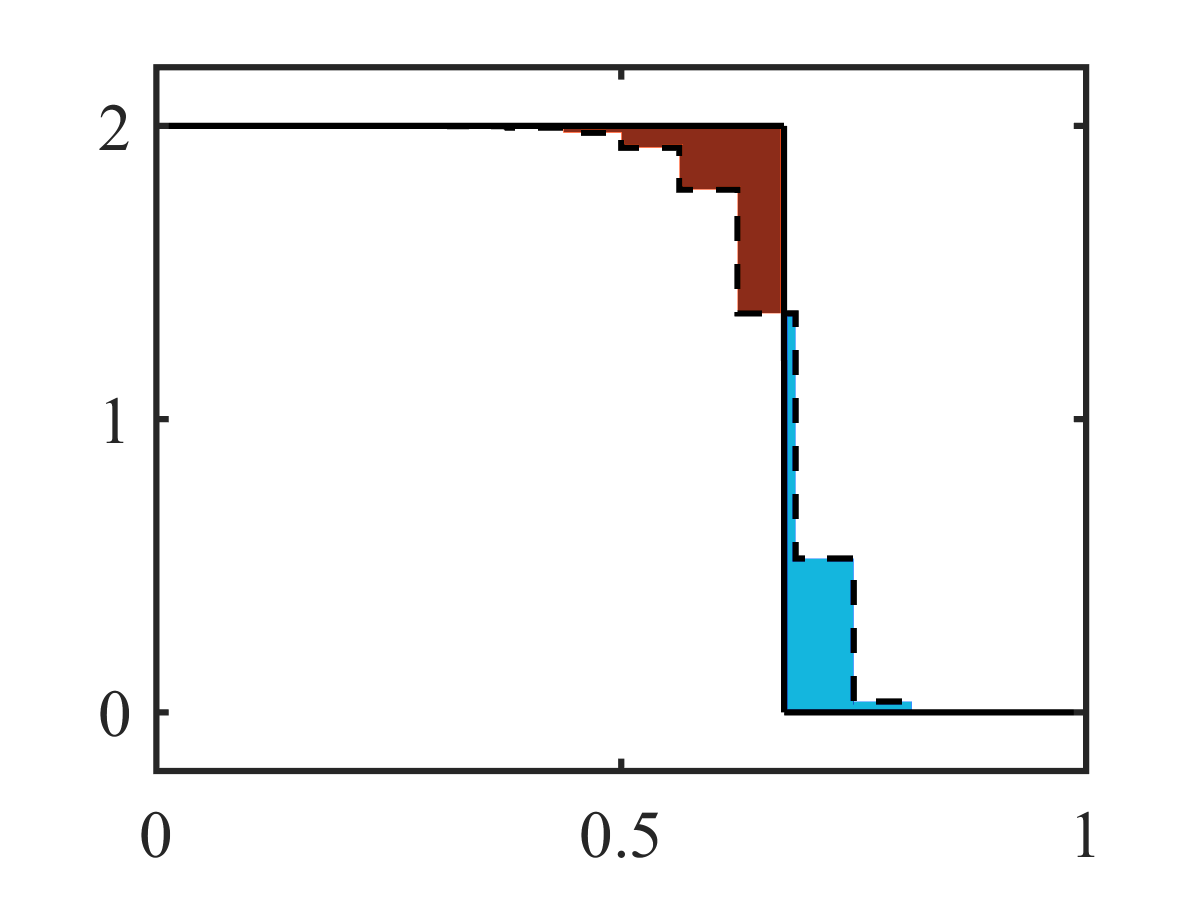}}
\caption{The initial condition \eqref{eq:numinitialcond} and the exact solution (solid curve) and numerical approximation (dashed curve) of \eqref{eq:burgers} before and after shock interaction. The color coding is the same as in Figure \ref{fig:wasserstein}.}
\label{fig:example}
\end{figure}

\begin{table}
\centering
\begin{tabular}{|c|c |c |c | c|}
\hline 
$ n $ & $L^1$ & $L^1$ OOC & $W_1$ & $W_1$ OOC \\ 
 \hline 
$ 32 $ & $ 4.078\times 10^{-2} $ & & $ 1.775\times 10^{-3} $ &  \\ 
$ 64 $ & $ 2.735\times 10^{-2} $ & $ 0.577 $ & $ 6.523\times 10^{-4} $ & $ 1.445 $ \\ 
$ 128 $ & $ 1.604\times 10^{-2} $ & $ 0.770 $ & $ 2.063\times 10^{-4} $ & $ 1.661 $ \\ 
$ 256 $ & $ 8.478\times 10^{-3} $ & $ 0.920 $ & $ 5.699\times 10^{-5} $ & $ 1.856 $ \\ 
$ 512 $ & $ 4.419\times 10^{-3} $ & $ 0.940 $ & $ 1.452\times 10^{-5} $ & $ 1.973 $ \\ 
$ 1024 $ & $ 2.121\times 10^{-3} $ & $ 1.059 $ & $ 3.632\times 10^{-6} $ & $ 1.999 $ \\ 
$ 2048 $ & $ 1.060\times 10^{-3} $ & $ 1.001 $ & $ 9.081\times 10^{-7} $ & $ 2.000 $ \\ 
$ 4096 $ & $ 5.341\times 10^{-4} $ & $ 0.989 $ & $ 2.270\times 10^{-7} $ & $ 2.000 $ \\ 
\hline 
\end{tabular}
\caption{Convergence rates for the Godunov scheme before the shock interaction, $t=0.15$}
\label{tab:god_before}
\end{table}

\begin{table}
\centering
\begin{tabular}{|c|c |c |c | c|}
\hline 
$ n $ & $L^1$ & $L^1$ OOC & $W_1$ & $W_1$ OOC \\ 
 \hline 
$ 32 $ & $ 2.848\times 10^{-2} $ & & $ 8.644\times 10^{-4} $ &  \\ 
$ 64 $ & $ 1.986\times 10^{-2} $ & $ 0.520 $ & $ 2.208\times 10^{-4} $ & $ 1.969 $ \\ 
$ 128 $ & $ 6.780\times 10^{-3} $ & $ 1.550 $ & $ 3.955\times 10^{-5} $ & $ 2.481 $ \\ 
$ 256 $ & $ 3.646\times 10^{-3} $ & $ 0.895 $ & $ 8.788\times 10^{-6} $ & $ 2.170 $ \\ 
$ 512 $ & $ 1.176\times 10^{-3} $ & $ 1.632 $ & $ 1.892\times 10^{-6} $ & $ 2.215 $ \\ 
$ 1024 $ & $ 9.863\times 10^{-4} $ & $ 0.254 $ & $ 5.291\times 10^{-7} $ & $ 1.838 $ \\ 
$ 2048 $ & $ 3.710\times 10^{-4} $ & $ 1.411 $ & $ 1.182\times 10^{-7} $ & $ 2.163 $ \\ 
$ 4096 $ & $ 2.255\times 10^{-4} $ & $ 0.718 $ & $ 3.308\times 10^{-8} $ & $ 1.837 $ \\ 
\hline 
\end{tabular}
\caption{Convergence rates for the Godunov scheme after the shock interaction, $t=0.3$}
\label{tab:god_after}
\end{table}

Tables \ref{tab:secondeno} and \ref{tab:thirdeno} show the convergence rate of the second- and third-order ENO-schemes using third-order Runge-Kutta time integration and, in space, the Godunov flux with second- and third-order ENO reconstructions at the cell boundaries. Although the existence of discrete shocks has not been proven for higher-order schemes such as the second- and third-order ENO-schemes, we also observe a first-order convergence rate in the $L^1$-norm and a second-order rate in the $W_1$-distance, as shown in Tables \ref{tab:secondeno} and \ref{tab:thirdeno}.

\begin{table}
\centering
\begin{tabular}{|c|c |c |c | c|}
\hline 
$ n $ & $L^1$ & $L^1$ OOC & $W_1$ & $W_1$ OOC \\ 
 \hline 
$ 32 $ & $ 2.125\times 10^{-2} $ & & $ 5.080\times 10^{-4} $ &  \\ 
$ 64 $ & $ 1.032\times 10^{-2} $ & $ 1.042 $ & $ 1.480\times 10^{-4} $ & $ 1.779 $ \\ 
$ 128 $ & $ 5.307\times 10^{-3} $ & $ 0.960 $ & $ 3.824\times 10^{-5} $ & $ 1.953 $ \\ 
$ 256 $ & $ 2.604\times 10^{-3} $ & $ 1.027 $ & $ 9.684\times 10^{-6} $ & $ 1.982 $ \\ 
$ 512 $ & $ 1.492\times 10^{-3} $ & $ 0.804 $ & $ 2.432\times 10^{-6} $ & $ 1.994 $ \\ 
$ 1024 $ & $ 6.553\times 10^{-4} $ & $ 1.187 $ & $ 5.965\times 10^{-7} $ & $ 2.027 $ \\ 
$ 2048 $ & $ 3.319\times 10^{-4} $ & $ 0.981 $ & $ 1.496\times 10^{-7} $ & $ 1.995 $ \\ 
$ 4096 $ & $ 1.628\times 10^{-4} $ & $ 1.028 $ & $ 3.783\times 10^{-8} $ & $ 1.984 $ \\ 
\hline 
\end{tabular}
\caption{Convergence rates for the second-order ENO scheme with Godunov flux at $t=0.15$}
\label{tab:secondeno}
\end{table}
\begin{table}
\centering
\begin{tabular}{|c|c |c |c | c|}
\hline 
$ n $ & $L^1$ & $L^1$ OOC & $W_1$ & $W_1$ OOC \\ 
 \hline 
$ 32 $ & $ 1.568\times 10^{-2} $ & & $ 3.454\times 10^{-4} $ &  \\ 
$ 64 $ & $ 6.516\times 10^{-3} $ & $ 1.267 $ & $ 8.128\times 10^{-5} $ & $ 2.087 $ \\ 
$ 128 $ & $ 3.528\times 10^{-3} $ & $ 0.885 $ & $ 2.104\times 10^{-5} $ & $ 1.950 $ \\ 
$ 256 $ & $ 1.696\times 10^{-3} $ & $ 1.056 $ & $ 5.286\times 10^{-6} $ & $ 1.993 $ \\ 
$ 512 $ & $ 9.825\times 10^{-4} $ & $ 0.788 $ & $ 1.329\times 10^{-6} $ & $ 1.992 $ \\ 
$ 1024 $ & $ 4.078\times 10^{-4} $ & $ 1.269 $ & $ 3.186\times 10^{-7} $ & $ 2.061 $ \\ 
$ 2048 $ & $ 2.205\times 10^{-4} $ & $ 0.887 $ & $ 8.219\times 10^{-8} $ & $ 1.955 $ \\ 
$ 4096 $ & $ 1.060\times 10^{-4}$ & $ 1.056 $ & $ 2.065\times 10^{-8} $ & $ 1.993 $ \\ 
\hline 
\end{tabular}
\caption{Convergence rates for the third-order ENO scheme with Godunov flux at $t=0.15$}
\label{tab:thirdeno}
\end{table}

\section{Conclusion and outlook}
\label{sec:conclusion}
With decreasing initial data consisting of a finite number of piecewise constants, we show that any monotone, $W_1$-contractive finite volume scheme will converge to the exact solution of \eqref{eq:cons_law} at a rate of $\Dx^2$ in the Wasserstein distance, both before and after shock interaction. The proof of the main result relies on the existence of discrete shocks and on the $W_1$-contractivity of the scheme. In Section \ref{sec:W1contractive} we show that the Lax--Friedrichs, Engquist--Osher and Godunov schemes are $W_1$-contractive, but a general result for all monotone schemes is ongoing work.

In addition to illustrating the main theorem in this paper, the numerical results in Section \ref{sec:numerical} show a second-order convergence rate in $W_1$ for both the second- and third-order ENO schemes in the MUSCL formulation. Fan \cite{fan00} has established existence of discrete shocks (for $D^k\lambda$ rational) for certain second-order MUSCL schemes. This suggests that it might be possible to prove existence of discrete shocks for a more general class of schemes and thus extend the main result of this paper. 

Our analysis only applies to monotone schemes, and to an admittedly simple class of entropy solutions. Nonetheless, the main result, together with the numerical results for the higher-order ENO schemes, strongly suggest that measuring the approximation error in the $W_1$-distance may be the path to proving higher-order convergence rates. However, numerical evidence show that the $O(\Delta x)$ convergence rate obtained by Nessyahu, Tadmor and Tassa \cite{nessy_conv} is optimal for general $\Lip^+$-bounded initial data. Hence, it is necessary to use (formally) higher-order methods in order to tackle general initial data. The $W_1$-convergence rates of (formally) second- and higher-order schemes for more general initial data will be investigated further. 

Finally, we mention that although the $W_1$-distance is easily defined also in several (spatial) dimensions, there are currently no results on the $W_1$-convergence of schemes for multi-dimensional conservation laws.

\end{document}